\documentclass{imanum}
\jno{drnxxx}
\received{15 October 2019}
\revised{}
\usepackage[utf8]{inputenc}
\usepackage[T1]{fontenc}
\usepackage{url}
\usepackage{amssymb}
\usepackage{amsmath}
\usepackage{mathrsfs} 
\usepackage{esint}
\usepackage{stmaryrd} 
\usepackage{phonetic} 
\usepackage{upgreek}
\usepackage{paralist} 
\usepackage{cancel}
\usepackage{textcomp}

\renewcommand{\emph}[1]{\textsl{#1}} 
\newcommand{\ML}{{\sc Matlab}}
\newcommand{\MP}{{\sc Maple}}

\newcommand{\diag}{\mathrm{diag}}
\renewcommand{\vec}[1]{\ensuremath{\bm{\mathbf{#1}}}}
\newcommand{\mat}[1]{\ensuremath{\bm{\mathbf{#1}}}}
\usepackage{bm}
\newcommand{\vecg}[1]{\ensuremath{\bm #1}}

\newcommand{\A}{\ensuremath{\mat{A}}}

\newcommand{\U}{\ensuremath{\mat{T}}}
\newcommand{\Z}{\ensuremath{\mat{Z}}}
\newcommand{\bary}{\ensuremath{\beta}}

\newcommand{\xder}[1]{\ensuremath{{{#1}^{\,\prime}}}}

\usepackage{xfrac} 
\newcommand{\defaultwidth}{.75\textwidth} 

\renewcommand{\textit}{\textsl}

\begin{document}
\title{Compact Finite Differences and Cubic Splines}
\shorttitle{Compact Finite Differences and Cubic Splines}
\author{Robert M.~Corless\thanks{Email: rcorless@uwo.ca} \\[2pt]
The Ontario Research Center for Computer Algebra
and The School of Mathematical and Statistical Sciences
The University of Western Ontario
London, Ontario,{Canada}
}

\maketitle

\begin{abstract}
{In this paper I uncover and explain---using contour integrals and residues---a connection between cubic splines and a popular compact finite difference formula.  The connection is that on a uniform mesh the simplest Pad\'e scheme for generating fourth-order accurate compact finite differences gives \textsl{exactly}
the derivatives at the interior nodes needed to guarantee twice-continuous differentiability for cubic splines.  
I also introduce an apparently new spline-like interpolant that I call a compact cubic interpolant; this is similar to one introduced in 1972 by Swartz and Varga, but has higher order accuracy at the edges.  I argue that for mildly nonuniform meshes the compact cubic approach offers some potential advantages, and even for uniform meshes offers a simple way to treat the edge conditions, relieving the user of the burden of deciding to use one of the three standard options: free (natural), complete (clamped), or ``not-a-knot'' conditions. Finally, I  establish that the matrices defining the compact cubic splines (equivalently, the fourth-order compact finite difference formul\ae) are positive definite, and in fact totally nonnegative, if all mesh widths are the same sign, for instance if the mesh is real and nodes are increasing only.}
{compact finite differences; cubic splines; barycentric form; compact cubic splines; contour integral methods; totally nonnegative matrices}
\end{abstract}




\section{Introduction}
\begin{flushright}
``The most popular choice continues to be a piecewise cubic approximating function.''\\
---Carl de Boor~\cite[p.~49]{deBoor(1978)}.
\end{flushright}

Cubic splines and the similar piecewise interpolant known as \texttt{pchip} are both widely used in piecewise polynomial interpolation.  Cubic splines give a twice continuously differentiable interpolant through given data, while \texttt{pchip} tries instead to preserve monotonicity and convexity.  Both are useful not just because they fit the data, but also because their derivatives are generally also good approximations to the derivative of the underlying function that one wants to approximate.  Loosely speaking, on a uniform mesh of width $h$ the fit of a cubic spline to a smooth function is accurate to $O(h^4)$ and the derivative is accurate to $O(h^3)$. For more careful error bounds see~\cite{swartz1972error} and~\cite{deBoor(1978)}.

In contrast to both cubic splines and \texttt{pchip}, a \textsl{true} cubic Hermite interpolant fits not only function values at the nodes but also true derivative values at the nodes: both cubic splines and \texttt{pchip} use substitutes computed from the function values.  The similarity of names (\texttt{pchip} stands for Piecewise Cubic Hermite Interpolant) does cause confusion. 
The relative rarity of the case when true derivatives $\rho_{k,1} = \xder{f}(\tau_k)$ are specified makes this confusion bearable. This paper will concentrate on cubic splines, and not on \texttt{pchip} or on true cubic Hermite interpolants. 
The classical reference for splines is~\cite{deBoor(1978)}, but see also chapter 3 of~\cite{Moler(2004)}.   

Compact finite differences are an efficient and accurate way of approximating the derivatives of known data.  There are several formul{\ae} in use.  See for instance~\cite{Lele(1992)} or~\cite{Collatz(1966)}.

In this paper I detail a connection between cubic splines on a uniform mesh and a popular compact finite difference formula that the traditional (monomial basis) method of computing the cubic splines had obscured.  Namely, the simplest Pad\'e scheme for generating fourth-order compact finite differences gives exactly
the derivatives at the interior nodes needed to guarantee twice-continuous differentiability; that is, a spline.
The literature both of splines and of finite differences is vast.  Nonetheless I believe that this connection is new.
In this paper I will give the two derivations, and show that, in the case of an equally-spaced mesh, equation~\eqref{eq:compactequispaced1} from the compact scheme and equation~\eqref{eq:spline} from the cubic spline are identical after normalization (each could be multiplied by an arbitrary constant without altering the solution). I then offer an explanation in section~\ref{sec:explanation} of why this is so (or why it might have been expected to be so).

I also comment on the various choices for treating the degrees of freedom at the endpoints, and introduce a new choice, which I call a ``compact cubic spline'', namely the use of a modified compact finite difference formula to give the derivatives at the endpoints.
I prove that the matrices involved are positive definite (indeed totally nonnegative) if the mesh widths have the same sign\footnote{Normally, the mesh will be real with the nodes increasing, so the mesh widths will all be positive.  Sometimes it is convenient to number real nodes from the rightmost endpoint, so the mesh widths will all be negative.  If the nodes are complex and not in a straight line, then indeed the mesh widths will not have the same (complex) sign.}.
This is similar in spirit to the approach of~\cite{swartz1972error}, who use a four-node explicit finite difference at each end to provide $O(h^3)$ accurate approximate derivatives at the edges of a uniform grid.  Here, we use a compact four-node formula that gives $O(h^4)$ accurate approximate derivatives at the edges, uniform grid or not, if the function being interpolated has at least five continuous derivatives.
\section{Notation}
We study two basic problems and their connection in this paper.
The first basic problem under consideration is construction of a piecewise polynomial interpolant $p(t)$ from function value data $\rho_{i,0}$ on a possibly nonuniform mesh made up of distinct nodes $\tau_i$, numbered $0 \le i \le n$.  This will lead to $n+1$ by $n+1$ matrices. We will always take $n \ge 4$ because one can fit a single degree $n$ polynomial on $n+1$ points, and if $n \le 3$ this can be done with a single cubic.

We postpone discussion of the second basic problem, which is to take derivatives of data given on a possibly nonuniform mesh, until section~\ref{sec:compact}. For simplicity in both problems, we restrict to the real case and suppose that values of a smooth function $f:[a,b] \to \mathbb{R}$ have been sampled on some partition of a finite interval $[a,b]$, with $a = \tau_0 < \tau_1 < \tau_2 < \cdots < \tau_{n-1} < \tau_n = b$.  We write $\rho_{i,0} = f(\tau_i)$ for the known function values, for $0 \le i \le n$. On each subinterval $[\tau_k,\tau_{k+1}]$ the piecewise interpolant will be given by $p(t) = p_k(t)$, a polynomial \textsl{of degree at most~$3$} satisfying $p_k(\tau_k) = \rho_{k,0}$ and $p_{k}(\tau_{k+1})=\rho_{k+1,0}$, and satisfying some other conditions that we will specify later.
We will refer to this as a piecewise cubic polynomial even if some or all of its degrees are less than~$3$.  

As is very well known, both cubic splines and the shape-preserving interpolant known as \texttt{pchip} construct values $\rho_{i,1}$ at the interior nodes $\tau_i$ for $1 \le i \le n-1$ of the \textsl{derivative} of each $p_k(t)$ in order to ensure continuous differentiability there: we wish $p_{i-1}'(\tau_{i}) = p_{i}'(\tau_{i})$ for all interior nodes, $1 \le i \le n-1$. A spline goes further than \texttt{pchip} and chooses these values in order to ensure \textsl{twice} continuous differentiability at the interior nodes: $p_{i-1}''(\tau_{i}) = p_{i}''(\tau_{i})$. By not imposing this last condition, a \texttt{pchip} can use some degrees of freedom to preserve monotonicity instead, usually by taking the derivatives $\rho_{i,1}$ to be a harmonic mean of the slopes of the secants on either side (see~\cite{Moler(2004)}).

\section{Cubic Splines\label{sec:spline}}
What follows in this section is a new derivation, or at least a derivation likely new to the reader, of the piecewise cubic Hermite spline. This derivation was apparently first published in Chapter~8 of~\cite{corless2013graduate}. The novelty is that it uses the barycentric forms for a cubic Hermite interpolant on each piece: that is, instead of trying to fit unknown cubic polynomials in a local monomial basis $p_k(t) = a_k + b_k(t-\tau_k) + c_k(t-\tau_k)^2 + d_k(t-\tau_k)^3$ to the data and trying to find reasonable ways to determine the $4n$ unknowns (there are $n$ subintervals), one instead works directly with
\begin{align}
p_k(t) = \frac{\displaystyle\sum_{i=k}^{k+1}\sum_{j=0}^1\sum_{\ell=0}^j\bary_{i,j}\rho_{i\ell}(t-\tau_i)^{\ell-j-1}}{\displaystyle\sum_{i=k}^{k+1}\sum_{j=0}^1\bary_{i,j}(t-\tau_i)^{-j-1}}\>,
\end{align}
which is the second barycentric form of the cubic Hermite interpolant
(note that only two nodes are used in this form and thus this determines only one piece of the interpolant). There are four unknowns in this formula: $\rho_{k,0}$, $\rho_{k,1}$, $\rho_{k+1,0}$, and $\rho_{k+1,1}$. 
The $\bary_{i,j}$ can be found from the partial fraction decomposition of the reciprocal of the node polynomial, and thus can be regarded as known once the nodes $\tau_k$ are specified:
\begin{multline*}
\frac{1}{(t-\tau_{k+1})^2(t-\tau_k)^2} = \sum_{i=k}^{k+1}\sum_{j=0}^1\bary_{i,j}(t-\tau_i)^{-j-1}\\
= \frac{\displaystyle\frac{-2}{(\tau_{k+1}-\tau_k)^3}}{t-\tau_k}+ \frac{\displaystyle \frac{1}{(\tau_{k+1}-\tau_k)^2}}{(t-\tau_k)^2} + \frac{\displaystyle \frac{2}{(\tau_{k+1}-\tau_k)^3}}{t-\tau_{k+1}}+\frac{\displaystyle\frac{1}{(\tau_{k+1}-\tau_k)^2}}{(t-\tau_{k+1})^2}\>.
\end{multline*}
There are only four $\bary_{i,j}$ for each interval, and we see them written above explicitly in terms of the given nodes $\tau_k$.
For convenience, one can simplify the second barycentric form to the usual cubic Hermite polynomial basis (here $h_{k+1} = \tau_{k+1}-\tau_k$ is the width of the
interval):
\begin{multline}
p(t) = {\frac { \left( t-\tau_{{k+1}} \right) ^{2} \left( 2\,t-2\,\tau_{{k}}+
h_{{k+1}} \right) \rho_{{k,0}}}{{h_{{k+1}}}^{3}}}+{\frac { \left( t-\tau_{
{k+1}} \right) ^{2} \left( t-\tau_{{k}} \right) \rho_{{k,1}}}{{h_{{k+1}}
}^{2}}}\\
+{\frac { \left( t-\tau_{{k}} \right) ^{2} \left( 2\,\tau
_{{k+1}}+h_{{k+1}} -2\,t\right) \rho_{{k+1,0}}}{{h_{{k+1}}}^{3}}}+{\frac {
 \left( t-\tau_{{k}} \right) ^{2} \left( t-\tau_{{k+1}}
 \right)\rho_{{k+1,1}} }{{h_{{k+1}}}^{2}}}
\>. \label{eq:cubichermitebary}
\end{multline}

\begin{remark}\label{rem:backstabHermite}
Evaluation of these formul\ae\ and their
derivatives does not entail significant rounding error.  A standard backward error analysis (not given here) shows
that if IEEE standard floating-point arithmetic is used, then (using the notation of~\cite{Higham(1996)} ) $\mathrm{fl}\left(p(x)\right)$ is the exact value of a polynomial going through the data $[\rho_{k,0}(1+\theta_6),\rho_{k,1}(1+\theta_6), \rho_{k+1,0}(1+\theta_5), \rho_{k+1,1}(1+\theta_5) ]$; that is the floating-point evaluation of $p(x)$ is the exact value of a
polynomial going through data that is at most six rounding errors different to the original data.  
\end{remark}

Notice that the $\rho_{i,0}$ (not the $\rho_{i,1}$, which represent derivative values) are the known data values. We want to choose the $n+1$ slopes $\rho_{i,1}$ to make the resulting interpolant as smooth as possible. We will see that we can make it $\mathscr{C}^2[\tau_0,\tau_n]$; that is, the second derivative will be continuous at each interior node. This makes the piecewise interpolant a cubic spline, albeit represented in a different basis.
Notice also that we may choose the $\rho_{i,1}$ in such a way that we automatically have $p(t)\in\mathscr{C}^1[\tau_0,\tau_n]$: just take the slope at the right end of one interval to be the same slope at the left end of the next.  This is very natural because $\rho_{i,1}$  is then interpreted as `the' slope at the node $\tau_i$ (indeed it would be somewhat unnatural to have different slopes on the left and right, though we could do that if we wanted).  Having made our interpolant continuously differentiable by this device, then $p_{k-1}'(\tau_k^-)=\rho_{k,1}=p'_k(\tau_k^+)$. To further ensure $p(t)\in\mathscr{C}^2[\tau_0,\tau_n]$, we want to make the second derivatives equal, \emph{i.e.},
\begin{align}
p''_{k-1}(\tau^-_k)=p''_k(\tau_k^+)\qquad k=1,2,\ldots,n-1\>.
\end{align}
I personally found this easier than the necessary algebra in the local monomial case. We would have
\begin{align}
p_k(t)=\rho_{k,0}+\rho_{k,1}(t-\tau_k)+c_k(t-\tau_k)^2+d_k(t-\tau_k)^3
\end{align}
and even to make the function just $\mathscr{C}^1$ we would have to impose the condition
\begin{align*}
p_k'(\tau^-_{k+1})=p'_{k+1}(\tau_{k+1}^+)\>,
\end{align*}
which isn't automatic; we would have to \emph{enforce}
\begin{align*}
\rho_{k,1}+2c_k(\tau_{k+1}-\tau_k)+3d_k(\tau_{k+1}-\tau_k)^2=\rho_{k+1,1}\>.
\end{align*}
For $\mathscr{C}^2$, we would have to enforce yet another condition, namely
\begin{align*}
2c_k+6d_k(\tau_{k+1}-\tau_k)=2c_{k+1}\>.
\end{align*}
Of course, this can be done, and the solution is even elegant.  These equations can be reduced to a \emph{tridiagonal} system of equations for the slopes $\rho_{k,1}$, and explicit formul\ae\ for the $c_k$ and $d_k$ are known once the slopes are known. The solution is shown in {\tt splinetx.m} in~\cite{Moler(2004)}.

But here, because we start with the Hermite interpolational basis, we have a simpler (but equivalent) task: just enforce the second derivative conditions.  To do this, we need a formula for the second derivative of the cubic Hermite interpolant at the nodes. A short computation by hand or with a computer algebra language such as \MP\ shows that
\begin{align}
p_{k-1}''(\tau_k^-)=\frac{2}{\tau_k-\tau_{k-1}}(2\rho_{k,1}+\rho_{k-1,1})-\frac{6}{(\tau_k-\tau_{k-1})^2}(\rho_{k,0}-\rho_{k-1,0})
\end{align}
and
\begin{align}
p_k''(\tau_k^+)=\frac{-2}{\tau_{k+1}-\tau_k}(2\rho_{k,1}+\rho_{k+1,1})+\frac{6}{(\tau_{k+1}-\tau_k)^2}(\rho_{k+1,0}-\rho_{k,0})\>.
\end{align}
Equating these at interior nodes $1\leq k\leq n-1$ gives $n-1$ equations constraining the slopes. 
Explicitly, these equations are
\begin{multline}
\frac{2}{\tau_k-\tau_{k-1}}\rho_{k-1,1}+4\left(\frac{1}{\tau_k-\tau_{k-1}}+\frac{1}{\tau_{k+1}-\tau_k}\right)\rho_{k,1} +\frac{2}{\tau_{k+1}-\tau_k}\rho_{k+1,1}\\
=\frac{6}{(\tau_{k+1}-\tau_k)^2}(\rho_{k+1,0}-\rho_{k,0})+\frac{6}{(\tau_k-\tau_{k-1})^2}(\rho_{k,0}-\rho_{k-1,0}) \label{eq:hermiteqns}
\end{multline}
for $k=1,2,3,\ldots,n-1$. There are $n-1$ equations. The structure of the resulting matrix is tridiagonal. These are equivalent to those of equation~(9) of~\cite[p.~54]{deBoor(1978)}.  It is convenient to introduce $h_{k} = \tau_k-\tau_{k-1}$ and $h_{k+1} = \tau_{k+1}-\tau_k$ and multiply equation~\eqref{eq:hermiteqns} by $h_kh_{k+1}/(h_k+h_{k+1})$ to get
\begin{multline}
    \frac{2h_{k+1}}{h_k+h_{k+1}}\rho_{k-1,1}+4\rho_{k,1} +\frac{2h_{k}}{h_k+h_{k+1}}\rho_{k+1,1}\\
=\frac{6h_{k}}{h_{k+1}(h_k+h_{k+1})}(\rho_{k+1,0}-\rho_{k,0})+\frac{6h_{k+1}}{h_{k}(h_k+h_{k+1})}(\rho_{k,0}-\rho_{k-1,0})\>. \label{eq:rescaledhermiteqns}
\end{multline}

If the mesh is equally-spaced, that is $\tau_k = a + (b-a)k/n = a + k h$ for $0 \le k \le n$, with $h = (b-a)/n$, then these equations become for $1 \le i \le n-1$
\begin{align}
    \rho_{i-1,1} + 4\rho_{i,1} + \rho_{i+1,1} &= \frac{3}{h}\left( \rho_{i+1,0} - \rho_{i-1,0}\right)\>. \label{eq:spline}
\end{align}
If we specify two edge conditions and then solve this tridiagonal system of $n-1$ equations, we will get derivative values at the interior nodes that will ensure a twice continuously differentiable piecewise interpolant, where each piece $p_k(t)$ is given by equation~\eqref{eq:cubichermitebary}.  

This leaves two degrees of freedom, which (counterintuitively) may be unwanted: the user will have to choose what to do with these two degrees of freedom.  Some common choices are, first, to ask for \textsl{three} times continuous differentiability at $\tau_1$ and at $\tau_{n-1}$: this is called the ``not-a-knot'' condition.  Second, one can specify arbitrary values of the derivatives at the ends, which is sometimes called a ``clamped'' spline; if the derivative values are \textsl{correct} for the function being interpolated, it is called a \textsl{complete} spline in~\cite{deBoor(1978)}.  Finally, one can ask for a second derivative of zero at the ends, called the ``natural'' spline.

We will see a new choice, actually a new set of choices, similar to a fifth choice  due to~\cite{swartz1972error} that is mentioned in~\cite{deBoor(1978)}, after the next section.

\section{Compact finite differences}
\label{sec:compact}
This section is adapted from chapter 11 of~\cite{corless2013graduate}.  The problem being addressed is the problem of computing derivatives of a function known only at discrete values; we are not necessarily interested in computing an interpolant.  One could use, for instance, simple finite differences and approximate $f'(\tau_k)$ by $(f(\tau_{k+1})-f(\tau_k))/(\tau_{k+1}-\tau_k)$. The main idea of a \textsl{compact} finite difference is that instead of using a single explicit finite difference formula to evaluate a derivative at a point, we have a whole mesh of function values and we wish to compute the derivatives at \textsl{all} the nodes.  This is quite like the case of a global interpolating polynomial where one constructs a differentiation matrix (see~\cite{Weideman(2000)}), and indeed we will have the equivalent of a differentiation matrix here; but it will not be explicitly formed.  Instead we will solve a banded linear system. 

The purpose of this paper is to show that the canonical example of a compact finite difference happens to have a significant relationship with cubic splines.  This fact does not seem to have been noticed before, although it is difficult to be sure, given the variation in nomenclature and the large number of works on piecewise polynomials of one kind or another.  Specifically, the relationship is as follows.

Let $\Delta$ be the difference operator $\Delta(f) = f(x+h)-f(x)$. Then the so-called \textsl{operator approach} to finite differences (which appears for instance in~\cite{boole1860treatise}) gives a relationship between the differentiation operator $D$ and $\Delta$, namely $D = \ln(1+\Delta)/h$.
Then the $(2,2)$ Pad\'e approximant for $\ln(1+\Delta)$, namely
\begin{equation}
   D \doteq {\frac {\Delta+{\Delta}^{2}/2}{1+\Delta+{\Delta}^{2}/6}}\>,
\end{equation}
gives rise to a compact finite difference scheme (also called a Pad\'e scheme)
that happens to be fourth-order accurate: apply the \textsl{denominator} $1+\Delta+{\Delta}^{2}/6$ to $\xder{f}$ and the \textsl{numerator} $\Delta + \Delta^2/2$ to $f$ and we find after shifting to center at $t$ that, up to terms of $O(h^4)$,
\begin{align*}
\frac{1}{6}\xder{f}(t-h)+\frac{2}{3}\xder{f}(t)+\frac{1}{6}\xder{f}(t+h) = \frac{1}{2h}\left(f(t+h)-f(t-h)\right)\>.
\end{align*}
We will discuss further the operator approach briefly in section~\ref{sec:other}.
This formula gives us a tridiagonal (whence ``compact'') system of equations for the unknown derivatives, \textsl{at all interior nodes}; again something special has to be done at the edges.  This tridiagonal system of equations turns out to be identical up to scaling---when the nodes are equally-spaced---to equations~\eqref{eq:hermiteqns}.

The idea of solving a system of equations to find our finite difference approximation to the derivatives may be unfamiliar.  The simplest finite difference formul\ae\ make straightforward linear combinations such as $(f(t+h)-f(t-h))/(2h)$ to approximate $f'(t)$.  But compact finite difference formul\ae\ are different. 
That is, instead of simply \emph{applying a formula} to a vector of function values to get a vector of derivative values,
we instead have to \emph{set up and solve} a linear system of equations for the unknown derivatives.
Having to solve equations is more complicated than just using a formula, but it has several
advantages.

To better understand where the system of equations for this formula comes from, we give another derivation.  Make the following simplifying assumptions, for the moment.
Suppose $\xder{f}(\tau_0)$ and $\xder{f}(\tau_n)$ are known (just to make it simple) and that $\tau_{k+1}-\tau_k=h$ is constant. Then fix attention on
one particular node, say $\tau_k$.  The formula above becomes, when multiplied by $6$ and putting $t-h= \tau_{k-1}$, $t = \tau_k$, and $t+h=\tau_{k+1}$,
\begin{align}
\xder{f}(\tau_{k-1})+4\xder{f}(\tau_k) +\xder{f}(\tau_{k+1}) = \frac{3}{h}(f(\tau_{k+1})-f(\tau_{k-1}))\>. \label{eq:compactequispaced1}
\end{align}
\textbf{This equation is exactly equation~\eqref{eq:spline}.} Recall that $\rho_{k,1}$ represents derivatives $f'(\tau_k)$ and $f(\tau_k) = \rho_{k,0}$.
Now we let $k$ vary over all the indices of the interior nodes, $1 \le k \le n-1$.  Each interior node gives us one equation.
Each equation only contains at most three of the unknown derivatives (and the equation for $k=1$ touches the known derivative $\xder{f}(\tau_0)$, while the equation for $k=n-1$
touches the known derivative $\xder{f}(\tau_n)$).  This gives us a tridiagonal linear system of equations to solve for the unknown derivatives $\xder{f}(\tau_k)$, $1 \le k \le n-1$.
Call the tridiagonal matrix $\A$.  Notice also that the right-hand side of the system involves linear combinations of the values of $f(\tau_k)$ at different nodes---these are supposed to be
known.  Call the (also tridiagonal) matrix that forms that combination, $\mat{B}$.  Note that $\mat{B}$ has a zero diagonal.
The system
$\A\vec{v}=\mat{B}\vecg{\rho}$ needs to be solved computationally to get the vector $\vec{v}$ of desired derivatives.

In \emph{effect} this computes the differentiation matrix $\mat{D}$ as $\A^{-1}\mat{B}$, but in practice one never explicitly computes $\A^{-1}$ because it is a full matrix.
Instead, of course, to compute $\mat{D}\vec{y}$, one solves $\A\vec{v} = \mat{B}\vec{y}$ for $\vec{v}$, which is formally $\A^{-1}\mat{B}\vec{y}$.

\begin{figure}[ht]
\begin{center}
\includegraphics[width=\defaultwidth]{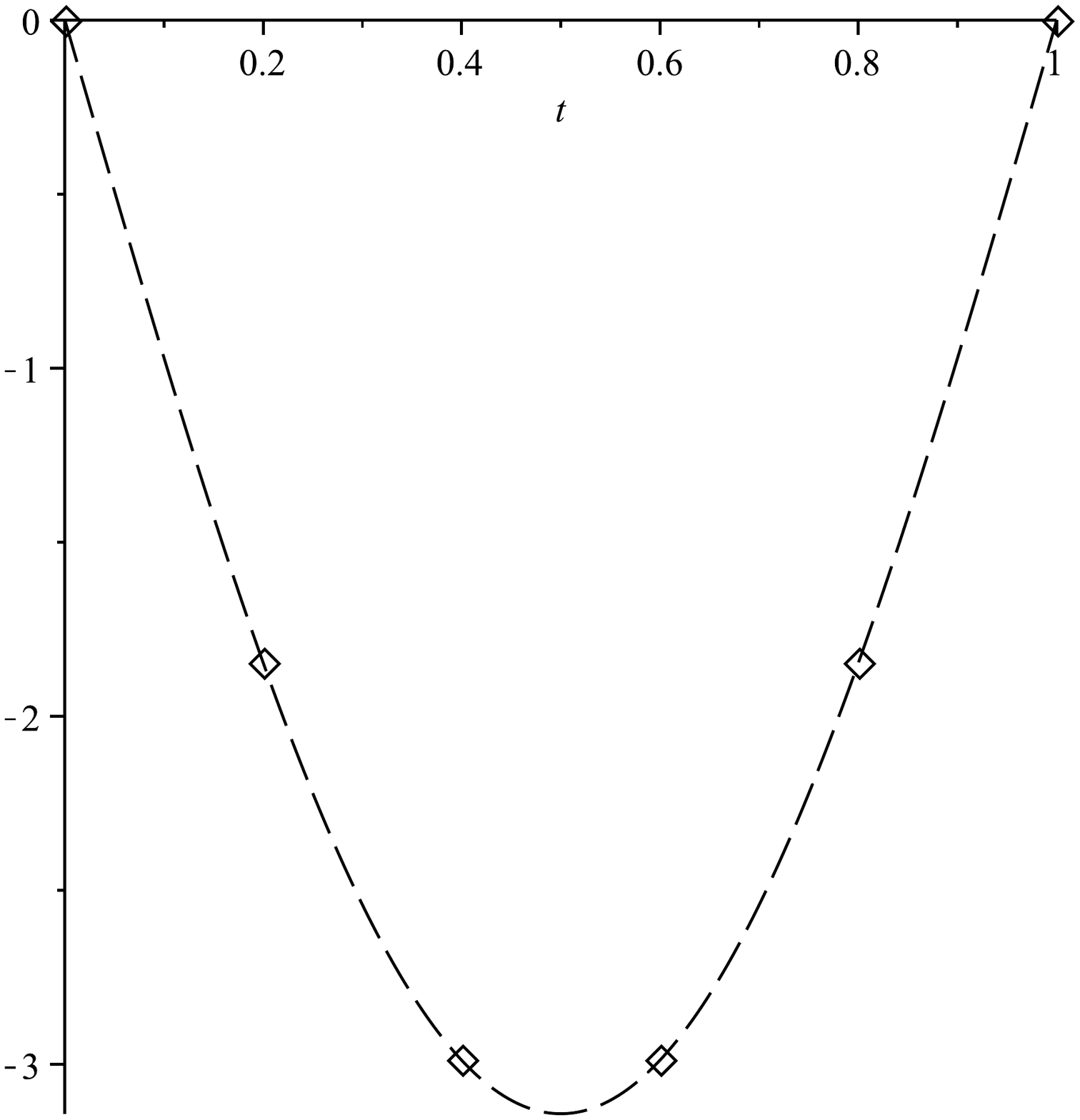}
\caption[4th-order uniform mesh compact finite difference derivative.]{\label{fig:compact4error} Fourth-order uniform mesh compact finite difference derivative graphed with the exact derivative (dashed line).  Even with just five subintervals, visual accuracy is achieved in this example.}
\end{center}
\end{figure}

I now give a derivation of a fourth-order compact finite difference method for nonuniform meshes, using contour integrals because this also works for the barycentric form.  This provides a point of theoretical continuity with section~\ref{sec:spline}, and will be useful for the explanation of the connection to cubic splines in section~\ref{sec:explanation}.  I will also give a similar derivation for a compact formula valid at the endpoints, which gives us a new way to deal with the extra degrees of freedom in a cubic spline. 

Consider the following partial fraction expansion: here we put $h_{k-1}=rh$ and $h_k = sh$ and we assume both $r$ and $s$ are different from zero (we could without loss of generality assume that one of $r$ or $s$ was $1$ but this doesn't help much, and I find the symmetry in the formul\ae\ below useful for understanding).
\begin{multline*}
\frac{1}{(z+rh)^2z^2(z-sh)^2} = \frac{1}{{r}^{2}{h}^{4} \left( s+r \right) ^{2} \left( z+rh \right) ^{2}}
+\frac{4\,r+2\,s}{ \left( z+rh \right){r}^{3}{h}^{5} \left( s+r \right) ^{3}}\nonumber\\
+\frac{1}{{r}^{2}{h}^{4}{s}^{2}{z}^{2}} + \frac{2\,r-2\,s}{z{r}^{3}{h}^{5}{s}^{3}}
+ \frac{1}{{h}^{4} \left( s+r \right) ^{2}{s}^{2} \left( z-sh \right)^{2}} + \frac{-4\,s-2\,r}{ \left( z-sh \right) {h}^{5} \left( s+r \right) ^{3}{s}^{3}}\>,
\end{multline*}
By a standard argument in complex variables, the contour integral of
\begin{equation}
\frac{1}{2\pi i}\oint_C \frac{f(z)}{(z+rh)^2z^2(z-sh)^2}\,dz
\end{equation}
over the contour that encloses all the zeros of the denominator (the numerator is a polynomial and has no poles) is zero for polynomials $f(x)$ of degree $6-2=4$ or less.  Once we expand this using the partial fraction above and use the Cauchy Integral formula
\begin{equation}
\frac{f^{(j)}(p)}{j!} = \frac{1}{2\pi i}\oint_C \frac{f(z)}{(z-p)^{j+1}}\,dz
\end{equation}
and multiply by $(s+r)^2$ for convenience,
this gives us the following formula, which is exact for polynomials of degree at most four\footnote{The fact that this is valid for polynomials of degree four is an important point in understanding why the equations are the same as for \textsl{cubic} splines, as we will see in section~\ref{sec:explanation}.}.
\begin{multline}
{\frac {1}{{r}^{2} }} \xder{f}(-rh) + {\frac {\left( s+r \right) ^{2}}{{r}^{2}{s}^{2}}} \xder{f}(0)
+ {\frac {1}{ {s}^{2}}}\xder{f}(sh)\\
= -{\frac {4\,r+2\,s}{{r}^{3}{h}(r+s)}} f(-rh)
-{\frac {2\left(r-s\right)\left( s+r \right) ^{2}}{{r}^{3}{h}{s}^{3}}} f(0)
+
{\frac {4\,s+2\,r}{{h}(r+s) {s}^{3}}} f(sh)\>. \label{eq:variablecompact}
\end{multline}
These equations, if $r\ne s$, are \textsl{not} the same as equations~\eqref{eq:rescaledhermiteqns}, whatever constant we use to normalize them.
Setting $r=s=1$, however,  gives equation~\eqref{eq:compactequispaced1}.
If we now have a mesh of distinct points $\tau_0 < \tau_1 < \cdots < \tau_{n-1} < \tau_n$, we can lay this compact formula first on $\tau_0$, $\tau_1$, and $\tau_2$,
with $rh = \tau_1 - \tau_0$ and $sh = \tau_2 - \tau_1$ which, if we have  a reference step width $h$, say $h = (\sum (\tau_{k+1}-\tau_k) )/n$,
gives us an equation relating the derivative values on these three mesh points to the function values on the mesh points.
We then lay the formula over $\tau_1$, $\tau_2$, and $\tau_3$, giving us another equation---one for each interior point, $\tau_1$, $\tau_2$,
$\ldots$, $\tau_{n-1}$.  The linear system for the unknown derivative values is, like the spline equations, tridiagonal; but we only have $n-1$ equations and $n+1$
unknowns $\xder{f}(\tau_0)$, $\xder{f}(\tau_1)$, $\ldots$, $\xder{f}(\tau_n)$.  To make a square system, we need two more equations. 

\subsection{Truncation error in the compact formula}
Once the formula has been found, one can do a more standard Taylor series analysis on it, for instance in \MP.   Taking the Taylor series in $h$ of the difference $L-R$ where $L$ is the left-hand side of equation~\eqref{eq:variablecompact} and $R$ is the right-hand side,  the error term turns out to be
\begin{multline}
L-R = {\frac { \left( s+r \right) ^{2}  }{120}}{h}^{4}f^{(5)}(0)-{\frac {
 \left( r-s \right)  \left( s+r \right) ^{2}  }{360}}{h}^{5}f^{(6)}(0)+
O(h^6)\>, \label{eq:compactfindiferror}
\end{multline}
which is $O(h^4)$ regardless of the local mesh ratios $r$ and $s$. If $r=s=1$ then this becomes equation~\eqref{eq:compactequispaced1} and the leading error term becomes $h^4/30$ times the fifth derivative of $f$ evaluated at zero.
\subsection{Truncation error in the spline formula, considered as a compact finite difference}
We can do the same kind of error analysis as we have just done for the compact formula, but for equations~\eqref{eq:rescaledhermiteqns}.  Put $\rho_{k,1} = f'(\tau_k)$ and $\rho_{k,0} = f(\tau_k)$ in equations~\eqref{eq:rescaledhermiteqns}. Take $h_{k-1} = rh$ and $h_k = sh$. For a sufficiently differentiable function $f$ we may take the Taylor series in $h$ of the difference $L-R$ between the left- and right-hand sides. Then the result is
\begin{equation}
    L-R = -\frac{rs(r-s)}{12}f^{(4)}(0) h^3 + \frac{rs(r^2-rs+s^2)}{30}f^{(5)}(0)h^4 + O(h^5)\>.
\end{equation}
This is the truncation error that occurs if the spline continuity equations~\eqref{eq:hermiteqns} are used as a compact finite difference formula. 
We see that if the mesh is not locally approximately uniform, that is $r \ne s + O(h)$, then this error is $O(h^3)$, not $O(h^4)$.  We may compare this with equation~\eqref{eq:compactfindiferror}; we see that if $r=s$ they are identical, as we will see that they must be.

\subsection{What to do at the edges}
We now return to computing fourth-order accurate compact finite differences.
We can use the same contour integral method to look for fourth-order
formulae at either end, giving equations involving $\tau_0$ and its nearest mesh neighbours, and $\tau_n$ and its nearest neighbours.  We will want
the formul\ae\ likewise exact for polynomials of degree four or less.  Since the matrix is so far tridiagonal, we try to keep it that way and we thus look
for relations of the form
\begin{equation}
a_0 \xder{f}(\tau_0) + b_0 \xder{f}(\tau_1) = c_0 f(\tau_0) + c_1 f(\tau_1) + c_2 f(\tau_2) + c_3 f(\tau_3)\>.\label{eq:compactvariablemesh}
\end{equation}
The reader may verify that fourth order accuracy is not possible in general without the function value at $\tau_3$.
This still qualifies as `compact' though because we use only two extra mesh points at the left end, and similarly only two extra on the right, and these appear in the right-hand side and do not change the tridiagonality of the matrix.  We will need $n \ge 4$, giving five by five matrices at the smallest.
This ansatz suggests looking at the partial fraction decomposition of
\begin{equation}
{\frac {1}{ \left( z-\tau_{{0}} \right) ^{2} \left( z-\tau_{{1}} \right) ^{2
} \left( z-\tau_{{2}} \right)  \left( z-\tau_{{3}} \right) }}\>,
\end{equation}
from which we straightforwardly find (of course by using a computer algebra system) that (with $h_1 = \tau_1-\tau_0$, $h_2 = \tau_2-\tau_1$, etc and normalizing so $b_0=1$)
\begin{align*}
a_{{0}}={\frac {h_2 \left( h_2+h_3 \right) }{ \left( h_1+h_2 \right)  \left( h_1+h_2+h_3 \right) }}
\quad\mathrm{and}
\quad
b_{{0}}=1\>,
\end{align*}
and that the $c_k$s are
\begin{align*}
c_0 &=-{\frac {h_2 \left( h_2+h_3 \right)  \left( 4\,{h_
1}^{2}+6\,h_1h_2+3\,h_1h_3+2\,{h_2}^{2}+2\,h_{
{1}}h_3 \right) }{h_1 \left( h_1+h_2 \right) ^{2}
 \left( h_1+h_2+h_3 \right) ^{2}}}
\\
c_1 &= {\frac {2\,h_2(h_2-h_1)+h_3(2h_2-h_1)}{h_1h_2 \left( h_2+h_3 \right) }}
\\
c_2 &= {\frac {{h_1}^{2} \left( h_2+h_3 \right) }{h_2
 \left( h_1+h_2 \right) ^{2}h_3}}
\\
c_3 &= -{\frac {{h_1}^{2}h_2}{h_3 \left( h_2+h_3
 \right)  \left( h_1+h_2+h_3 \right) ^{2}}}
\>.
\end{align*}
It turns out that the residual error in~\eqref{eq:compactvariablemesh} is, as desired, $O(h^4)$.  In detail, if $\tau_{k+1} = \tau_k + r_{k+1} h$, for $k=0$, $1$, and $2$, and $h_1 = \tau_1-\tau_0 = h$,
then the residual error is seen by a Taylor series computation to be 
\begin{equation}
L-R = \frac{(r_1+r_2)r_1}{120} h^4f^{(5)}(0) + O(h^5)\>.
\label{eq:leftenderror}
\end{equation} 
In the uniform mesh case this reduces to $h^4/60$ times the fifth derivative.  A similar formula holds for the other end (indeed, simply reverse the labels, $\tau_{k} \leftrightarrow \tau_{n-k}$).
This gives us closure in our search for a compact, variable mesh
fourth-order finite difference formula.  We are left with a tridiagonal matrix $\A$ with entries depending on the $h_k = \tau_{k}-\tau_{k-1}$.

Apart from the first and last rows, this matrix is diagonally dominant because $4 - 4(h_k^2+h_{k-1}^2)/(h_k+h_{k-1})^2 = 8h_kh_{k-1}/(h_k+h_{k-1})^2$.  We will see that it is indeed positive definite if the mesh widths have the same sign.
For meshes with widely varying widths, however, the matrix can be ill-conditioned.  See figure~\ref{fig:RandomMeshConditionNumbers}. 

By direct computation with Maple for small examples, we see that the determinant is a rational polynomial in all the $h_k$ with positive coefficients in the numerator and a squared denominator, and indeed is positive definite. Monitoring the sign of the determinant of $\A$ as computed numerically by Matlab's \texttt{det} function, we find that for $n\ge 4$ the determinant is always reported as being positive.  In fact, more is true.
\begin{theorem} \label{thm:totallynonnegative}
If all mesh widths $h_k$ have the same sign, then for $n\ge 4$ 
the $n+1$ by $n+1$ matrix $\A$ of the compact finite difference formula with these fourth order formul\ae\ at the edges 
is \textsl{totally nonnegative}: that is, the determinants of \textsl{all minors} are nonnegative.  More, the determinants of all principal minors are strictly positive, so the matrix is positive definite.
\end{theorem}
\begin{proof}
Without loss of generality we will assume all $h_k>0$.
By a theorem of Gantmacher and Krein (see~\cite[Chapter 29]{Hogben(2006)}), an irreducible tridiagonal matrix is totally nonnegative if and only if its entries are nonnegative and its leading principal minors are nonnegative.
Here all entries are nonnegative.
To prove the theorem we must then verify only that all leading principal minors are positive.  We do so by induction, for all but the last, and then handle that specially.


For simplicity, write $\A=\Z\U^{n+1,n+1}$ where $\Z = \diag(1,4/(h_1+h_2)^2, 4/(h_2+h_3)^2, \ldots, 4/(h_{n-1}+h_{n})^2,1)$ to remove the common denominator of each of the interior rows.  Multiplying by positive diagonal matrices leaves the properties of total positivity (nonnegativity) undisturbed. This leaves each interior row as $h_{k}^2, (h_k^2 + h_{k-1})^2, h_{k-1}^2$, for easier manipulation. Explicitly writing the six-by-six case, $\U^{6,6} = $
\begin{equation*}
    \left[ \begin {array}{cccccc} {\frac {h_{{2}} \left( h_{{3}}+h_{{2}}
 \right) }{ \left( h_{{2}}+h_{{1}} \right)  \left( h_{{3}}+h_{{2}}+h_{
{1}} \right) }}&1&0&0&0&0\\ \noalign{\medskip}{h_{{2}}}^{2}& \left( h_
{{2}}+h_{{1}} \right) ^{2}&{h_{{1}}}^{2}&0&0&0\\ \noalign{\medskip}0&{
h_{{3}}}^{2}& \left( h_{{3}}+h_{{2}} \right) ^{2}&{h_{{2}}}^{2}&0&0
\\ \noalign{\medskip}0&0&{h_{{4}}}^{2}& \left( h_{{3}}+h_{{4}}
 \right) ^{2}&{h_{{3}}}^{2}&0\\ \noalign{\medskip}0&0&0&{h_{{5}}}^{2}&
 \left( h_{{4}}+h_{{5}} \right) ^{2}&{h_{{4}}}^{2}
\\ \noalign{\medskip}0&0&0&0&1&{\frac {h_{{4}} \left( h_{{3}}+h_{{4}}
 \right) }{ \left( h_{{4}}+h_{{5}} \right)  \left( h_{{3}}+h_{{4}}+h_{
{5}} \right) }}\end {array} \right] \>.
\end{equation*}
Let $\U^{k,n+1} = \U[1{:}k,1{:}k]$ denote the leading principal submatrices. Let $D^{k,n+1} = \det(\U^{k,n+1})$ for $1 \le k \le n+1$.
By direct computation,
\begin{align}
    D^{1,n+1} &= {\frac {h_{{2}} \left( h_{{3}}+h_{{2}} \right) }{ \left( h_{{2}}+h_{{1
}} \right)  \left( h_{{3}}+h_{{2}}+h_{{1}} \right) }}
\nonumber \\
    D^{2,n+1} &= {\frac {h_{{1}}h_{{2}}h_{{3}}\left( h_{{2}}+h_{{1
}} \right)}{\left( h_{{2}}+h_{{1
}} \right)  \left( h_{{3}}+h_{{2}}+h_{{1}} \right)}}\nonumber \\
    D^{3,n+1} &= {\frac {h_{{1}}h_{{3}}{h_{{2}}}^{2} \left( h_{{3}}+h_{{2}} \right)\left( h_{{3}}+h_{{2}}+h_{{1}} \right) }{\left( h_{{2}}+h_{{1
}} \right)  \left( h_{{3}}+h_{{2}}+h_{{1}} \right)}}
 \label{eq:firstminorsU}
\end{align}
\begin{remark} If $n=3$, then $D^{4,4}=0$.   This is excluded from the theorem; as discussed previously, in that case, just fit a polynomial to all the data.
\end{remark}

Returning to the proof, we see that $D^{1,n+1} > 0$, $D^{2,n+1}>0$,  and $D^{3,n+1} > 0$, for any choice of positive~$h_k$. 

Laplace expansion about the last row of $\U^{k,n+1}$ for $3 \le k \le n$ gives a recursive formula for the determinant:
\begin{equation}
    D^{k,n+1} = (h_{k-1}+h_k)^2D^{k-1,n+1} - h_{k-2}^2h_k^2 D^{k-2,n+1}\>. \label{eq:interiordeterminants}
\end{equation}
For the last minor, namely the determinant of the full matrix, we have
\begin{equation}
D^{n+1,n+1} = \frac{h_{n-1}(h_{n-1}+h_{n-2})}{(h_{n-1}+h_{n})(h_{n}+h_{n-1}+h_{n-2})}D^{n,n+1} - h_{n-1}^2D^{n-1,n+1}\>. \label{eq:fulldeterminant}
\end{equation}
By direct computation,
\begin{align}
    D^{3,n+1} &= h_2^2D^{2,n+1} + q_2 \label{eq:firstQ}
\end{align}
where $q_2$ is the positive function
\begin{equation}
    q_2 = \frac{h_1h_2^2h_3^2(h_1+2h_2+h_3)}{(h_1+h_2)(h_1+h_2+h_3)}\>.
\end{equation}
Now assume inductively that $D^{k-1,n+1} = h_{k-2}^2D^{k-2,n+1} + q_{k-2}$ for a positive function $q_{k-2}$. We see above that this is true for $k-1 = 3$, or $k=4$. We will simultaneously establish by induction that $D^{k,n+1} > 0$ and that $D^{k,n+1} = h_{k-1}^2D^{k-1,n+1} + q_{k-1}$ for another positive function $q_{k-1}$, for $k \le n$.
\begin{remark}
The assumption $n\ge 4$ of the theorem is used here: we cannot start the induction otherwise. If $n=3$ then $D^{4,4}$ is not given by the interior formula, but by the end formula.
\end{remark}

Returning to the proof, by the interior  recurrence relation above,
\begin{align}
    D^{k,n+1} &= (h_k+h_{k-1})^2 D^{k-1,n+1} -h_{k}^2 h_{k-2}^2 D^{k-2,n+1} \nonumber \\
    &= h_{k-1}^2D^{k-1,n+1} + q_{k-1}
\end{align}
where, by the inductive assumption and expanding
$(h_k+h_{k-1})^2$,
\begin{equation}
    q_{k-1} = 2h_kh_{k-1}D^{k-1,n+1} + h_k^2q_{k-2}
\end{equation}
is clearly positive, being the sum of positive terms. 

This establishes by induction that all leading minors $D^{k,n+1}$ up to, but not including, the final determinant $D^{n+1,n+1}$, are positive. We verify that the final determinant is positive by a different method. 

If the $(n+1,n+1)$ entry in $\U^{n+1,n+1}$ were $\alpha$, then the determinant $D^{n+1,n+1}$ would be positive if
\begin{equation}
    \alpha D^{n,n+1} - h_{n-1}^2 D^{n-1,n+1} >0\>;
\end{equation}
or, on rearranging,
\begin{align}
    \alpha &> \frac{h_{n-1}^2 D^{n-1,n+1}}{D^{n,n+1}} \nonumber \\
    &> \frac{h_{n-1}^2}{ (h_{n}+h_{n-1})^2 - h_{n}^2h_{n-2}^2 D^{n-2,n+1}/D^{n-1,n+1}} 
\end{align}
by equation~\eqref{eq:interiordeterminants}.  Examining the ratio $D^{n-2,n+1}/D^{n-1,n+1}$ will lead to the \textsl{sufficient} condition
\begin{equation}
    \alpha > \frac{h_{n-1}^2}{ (h_{n-1}+h_{n})^2 - h_n^2h_{n-2}/(h_{n-2}+2h_{n-1})}
\end{equation}
for $D^{n+1,n+1}$ to be positive. 
In detail, because $D^{n-2,n+1}$ and $D^{n-1,n+1}$ are interior determinants, they satisfy
\begin{equation}
    \frac{D^{n-1,n+1}}{D^{n-2,n+1}} = h_{n-2}^2 + 2h_{n-2}h_{n-1} + \frac{h_{n}^2 q_{n-2}}{D^{n-2,n+1}}
\end{equation}
where that last term is positive. Dropping the last term and replacing $=$ by $>$ then following the manipulations (inversion, negation, inversion) through the algebra of inequalities gives the \textsl{upper} bound
\begin{multline}
    \frac{h_{n-1}^2}{ (h_{n}+h_{n-1})^2 - h_{n}^2h_{n-2}^2 D^{n-2,n+1}/D^{n-1,n+1}} \\
    < \frac{h_{n-1}^2}{ (h_{n}+h_{n-1})^2 - h_{n}^2h_{n-2}^2/(h_{n-2}^2 + 2h_{n-2}h_{n-1})}.
\end{multline}
Verifying that $\alpha = \frac{h_{n-1}(h_{n-1}+h_{n-2})}{(h_{n-1}+h_{n})(h_{n}+h_{n-1}+h_{n-2})} $ is larger than the right hand side, and thus larger than the left, is straightforward: after some algebra we have
\begin{equation}
    \alpha - \mathrm{rhs} = 
{\frac {h_{{n-1}}h_{{n-2}}h_{{n}} \left( h_{{n-2}}+2\,h_{{n-1}}+h_{{n}
} \right) }{ \left( h_{{n-1}}+h_{{n}} \right)  \left( h_{{n}}+h_{{n-1}
}+h_{{n-2}} \right) A }}
\end{equation}
where 
\begin{equation}
    A =  2\,{h_{{n}}}^{2}+2\,h_{{n}}h_{{n-2}}+4\,h_
{{n}}h_{{n-1}}+h_{{n-2}}h_{{n-1}}+2\,{h_{{n-1}}}^{2} 
\end{equation}
which, being positive, completes the theorem.
{\hfill$\natural$}
\end{proof}

\begin{figure}
    \includegraphics[width=\defaultwidth]{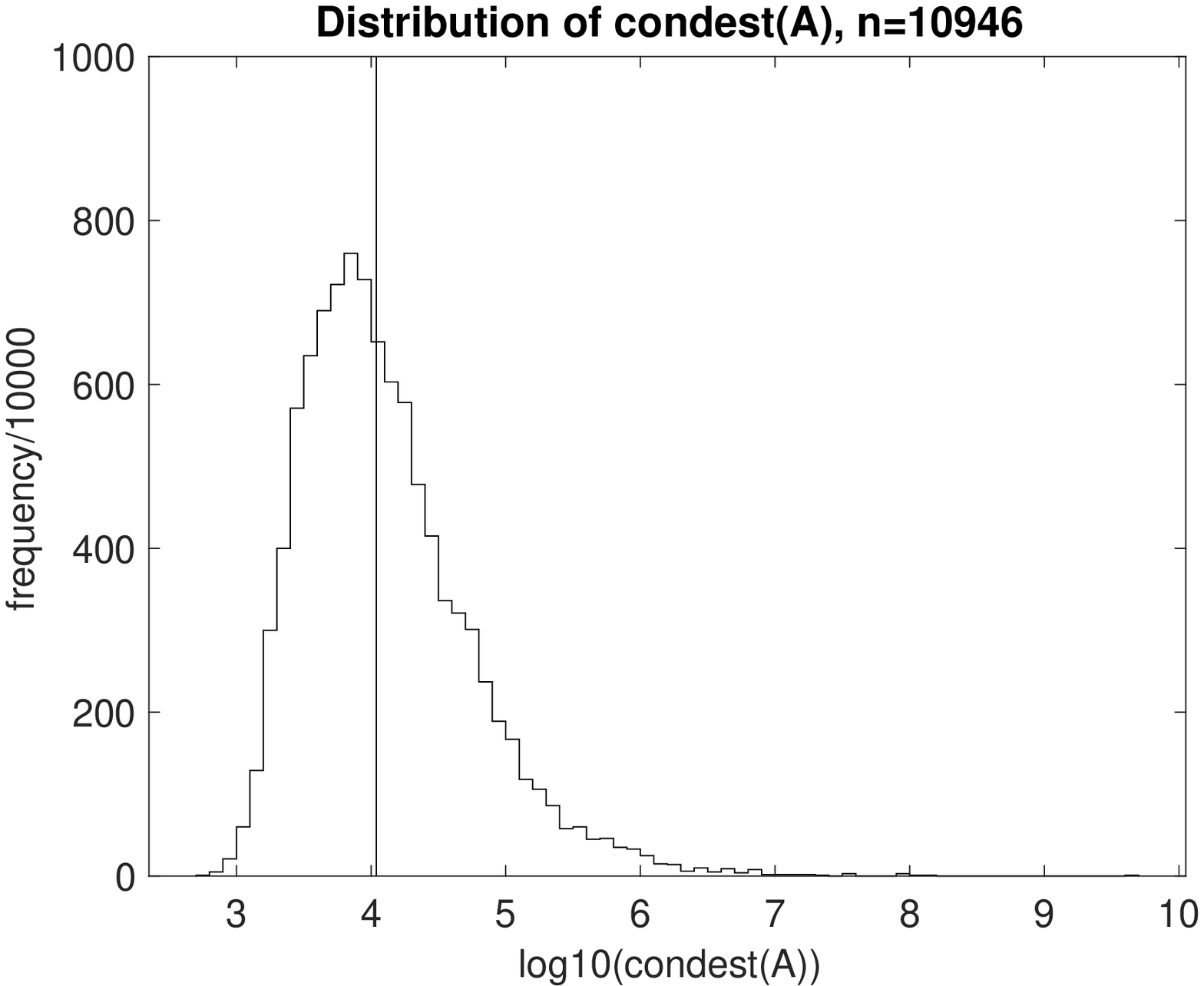}
    \caption{Condition numbers of the tridiagonal compact finite difference matrix for random meshes.  We sampled $10,000$ meshes of dimension $n=10,946$ (a Fibonacci number). The meshes were generated by \texttt{h = rand(1,n)}. The vertical line is at the dimension, $n$, suggesting that the mean condition number grows like $n$. Similar graphs were generated for other $n$, up to $n=46,368$. Plot generated by the command \texttt{ histogram( log(condsR(:,1))/log(10),'DisplayStyle','stairs')}.}
    \label{fig:RandomMeshConditionNumbers}
\end{figure}

\subsubsection{Some details of the equally-spaced case}
In the equally-spaced case the $(1,1)$ and $(n+1,n+1)$ entries each become $1/3$, the diagonal entries are $4$ otherwise, and all sub and superdiagonal entries are $1$.  LU factoring the matrix gives bidiagonal $\mat{L}$ and $\mat{U}$.  The diagonal entries of $\mat{U}$ are $1/3$, $1$, $3$, $11/3$, $41/11$, and so on up until the penultimate entry, for $2 < k < n-2$ with the form $a_{k+1}/a_k$ where $a_{k+1} = 4a_k - a_{k-1}$, and initial conditions $a_0 = a_1 = 1$ not reflected in the first diagonal entry. This is easily proved by induction.  The final entry in $\mat{U}$, call it $u_{n+1,n+1}$ must have $a_{n-1}/a_{n}\cdot 1 + u_{n+1,n+1} = 1/3$ (the subdiagonal entries in $L$ are the reciprocals of the diagonal entries of $\mat{U}$) and so $u_{n+1,n+1} = 1/3-a_{n-1}/a_{n}$. Solving the recurrence relation for the $a_k$ we have
\begin{equation}
    a_k = \left(\frac{1}{2} - \frac{\sqrt{3}}{6}\right) \left(2 + \sqrt{3}\right)^k + \left(\frac{1}{2} + \frac{\sqrt{3}}{6}\right) \left(2 - \sqrt{3}\right)^k  
\end{equation}
so $1/3-u_{n+1,n+1}$ very rapidly approaches $2-\sqrt{3} = 1/(2+\sqrt{3}) = 0.2677949$; indeed already $11/41 = 0.26829$ and thereafter $a_k/a_{k+1} = 1/(2+\sqrt{3}) + O(c^{-k})$ where $c \approx 13.9$.  This is important because the final entry $u_{n+1,n+1}$ is then very close to $1/3 - 1/(2+\sqrt{3}) \approx 0.06538 $.

These details can be used to establish the following, which I first observed as an experimental fact:
\begin{theorem}
The $1$-norm condition number of the $n+1$ by $n+1$ matrix $\A$ for uniform mesh is asymptotically $63 + 36\sqrt{3} \approx 125.35$, exponentially quickly independent of $n$.
\end{theorem}
\begin{proof}
Omitted for brevity.
\end{proof}

\subsubsection{Yet another alternative treatment of the endpoints}
In~\cite{Zhao(2007)} we didn't use equation~\eqref{eq:leftenderror}.  Instead we added an extra node $\tau_4$, which requires $n>3$, in order to specify the ratio $a_0/b_0 = c := 2+\sqrt{3}$. This allows analytical factoring of the matrix, speeding the process up slightly (it lowers the constant in the $O(n)$ cost). In detail, we found the edge formula
\begin{align}
    c f'(0) + f'(h) = -\frac {1}{h} \bigg( & \left( -25c/12-1/4 \right) f \left( 0 \right) \nonumber\\
& + \left( 4\,c-5/6 \right) f
 \left( h \right) + \left( -3\,c+3/2 \right) f \left( 2\,h \right) \nonumber \\
 &+
 \left( 4c/3-1/2 \right) f \left( 3\,h \right) + \left( -c/4+1/12
 \right) f \left( 4\,h \right)  {\bigg)}
\end{align}
(and a similar formula at the right edge with the sign of $h$ reversed but where $c = 4$)
which has truncation error
\begin{equation}
    \frac{(4c-1)h^4}{20} f^{(v)}(0) + \frac{(40c-11)h^5}{120} f^{(vi)}(0) + \cdots\>.
\end{equation}
This gives compact finite differences, again fourth order accurate. Note that the truncation error above is larger than the truncation error for the four-node formula; this seems to be compensated for by the fact that the condition number is smaller, by a similar amount.  

A similar formula is available for a nonuniform mesh. However, it does not have the same speed advantage that it does for a uniform mesh, because the matrix cannot (so far as I know) be analytically factored. The proof of Theorem~\ref{thm:totallynonnegative} is easily modified to show that the resulting matrix (for nonuniform meshes) is also totally nonnegative for positive mesh widths; in fact, this follows from a general fact about totally nonnegative matrices that says that the $(1,1)$ entries  and the $(n+1,n+1)$ entries can always be \textsl{increased} without altering the total nonnegativity (see~\cite[Chapter 29]{Hogben(2006)} for a statement and reference).

As mentioned, the advantage of this formula for a uniform mesh is that the tridiagonal matrix $\A$ can be \textsl{analytically} factored into $\mat{L}$ and $\mat{U}$, where independent of dimension the subdiagonal elements of $\mat{L}$ are all $1/c = 1/(2+\sqrt{3})$, and the diagonal elements of $\mat{U}$ are all $c$.  This saves the $O(n)$ cost of factoring the matrix, both in terms of operation count and in terms of storage.  The relatively trivial disadvantage of this formula is that it needs one more function value ($f(4h)$) on the right hand side. This makes no practical difference because in the only case where this matters, $n=3$, one can instead use a single cubic.
For the uniform mesh formula, the $1$-norm condition number is again asymptotically constant, but now being less than $3$ irrespective of the dimension; however, the truncation error at the edges is larger, and thus this advantage does not seem to matter.
\subsection{Nonuniform mesh compact derivatives/spline in practice}
So, how well do these methods work in practice?  Actually, pretty well.  As expected, if the mesh ratios are not ``too large,'' \emph{i.e.}, adjacent subintervals are not too different in width (so that the $r_k$ factors
in the edge formul\ae\ and the $r$ and $s$ factors in the interior formula are not too large), then the formula is very similar to the cubic spline and its derivative, for smooth functions.  As the mesh ratios depart more seriously from uniformity, the compact method begins to be more accurate than the cubic spline.

Note, however, that the formul\ae\ as initially written in terms of the $\tau_k$ are quite likely to produce rounding errors, especially in the formul\ae\ at the edges,
and should be rewritten using $h_{i+1} = \tau_{i+1}-\tau_i$ and factored, as done in the six-by-six example above, wherever
possible.  When this is done the influence of rounding errors, while still felt, is significantly reduced.
These formul\ae\ have been implemented in \ML\ in a program called \texttt{vcompact4}, available on the code repository for~\cite{corless2013graduate}.  
For this paper, I have written a program called \texttt{compactcubic}, which I discuss later.


\begin{figure}[ht]
\begin{center}
\includegraphics[width=\defaultwidth]{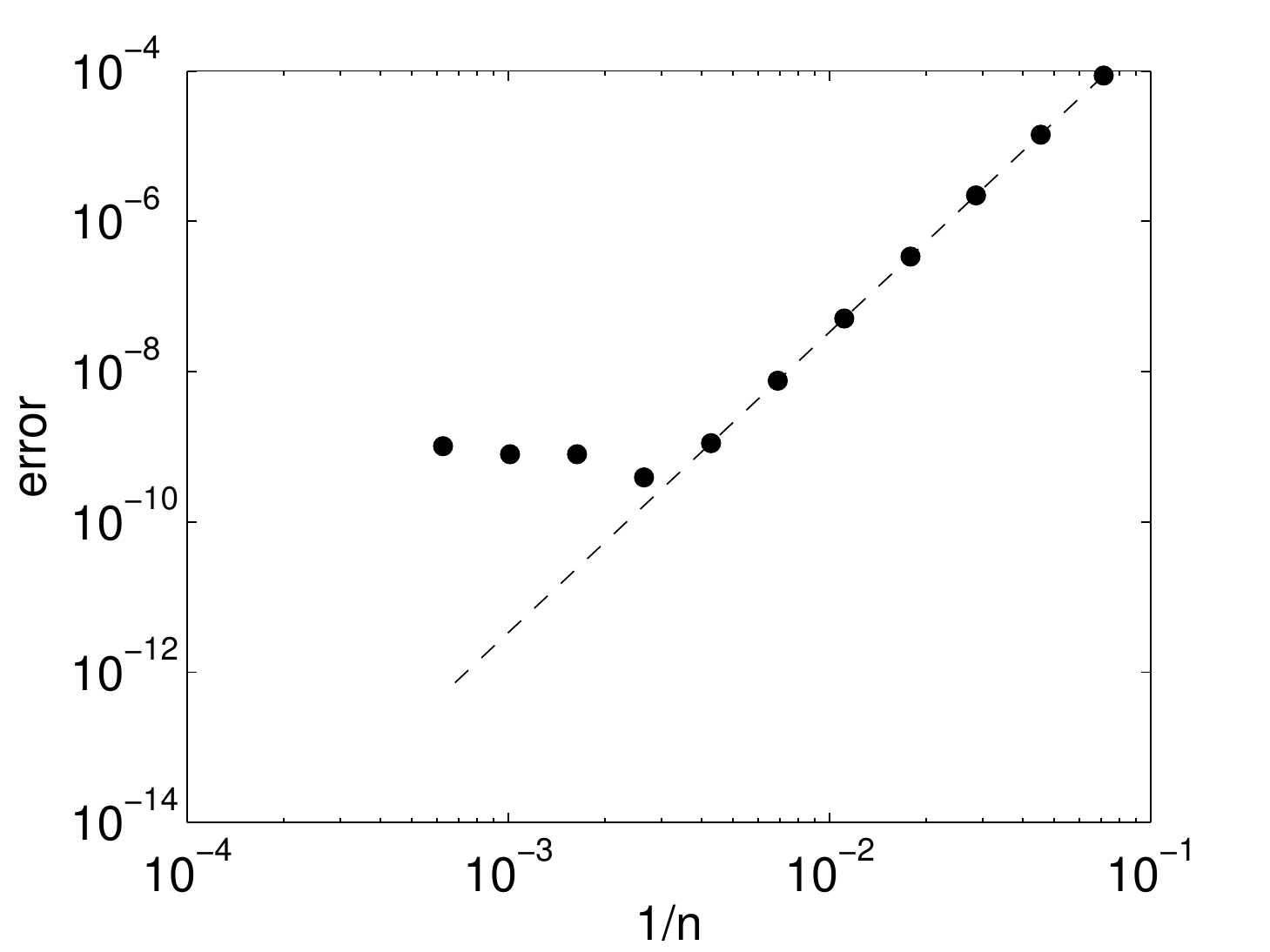}
\caption[4th-order nonuniform mesh compact finite difference maximum error.]{\label{fig:vcompact4} Fourth-order nonuniform mesh compact finite difference maximum error computing $(\sfrac{1}{\Gamma})'(x)$ on $1\le x \le 3$.
The spatial mesh was Chebyshev points $x_j = 2 + \cos(\sfrac{\pi j}{n})$ for $0 \le j \le n$, for various~$n$.  Theoretical fourth-order behaviour
is shown by the dashed line.}
\end{center}
\end{figure}
\begin{figure}
    \centering
    \includegraphics[width=\defaultwidth]{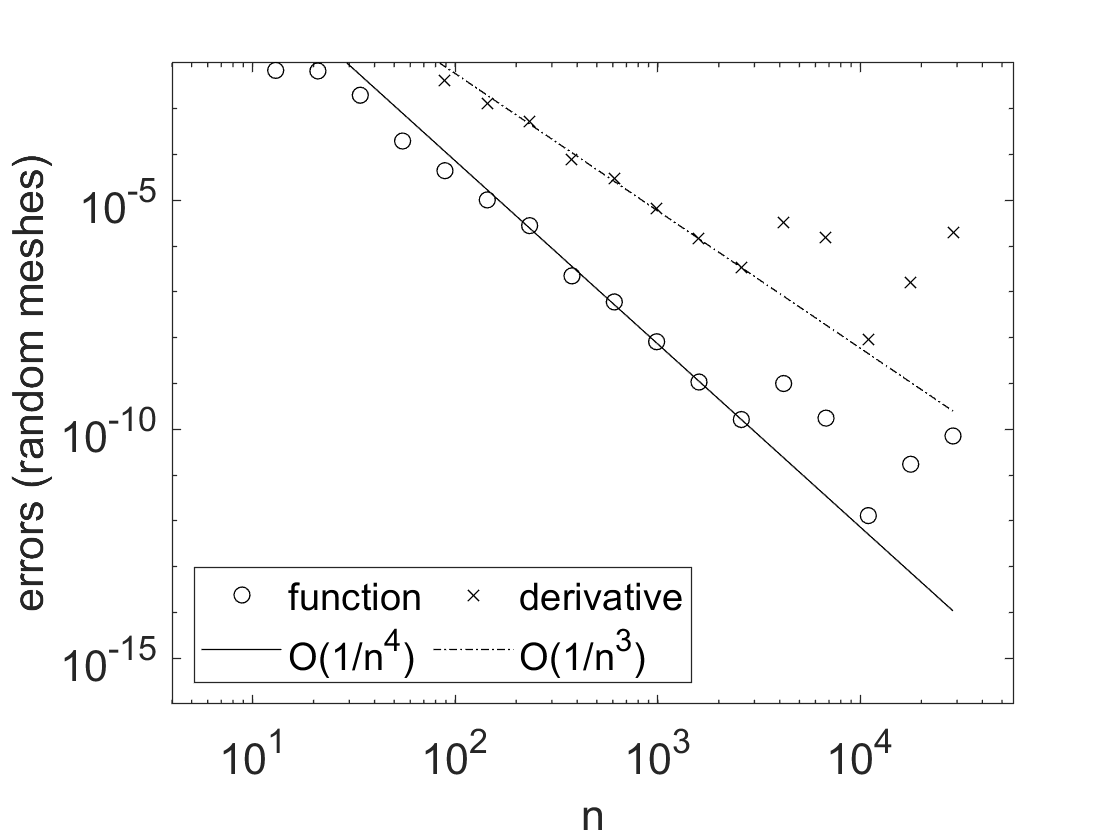}
    \caption[4th-order nonuniform mesh compact finite difference maximum error]{\label{fig:RandomMeshesRunge} Error in fourth-order compact cubic spline and its derivative for the Runge function on random meshes on $[-1,1]$. The random meshes were generated by \texttt{h = rand(1,n)} and then using a cumulative sum to generate grid points. The mesh widths varied in each individual mesh by factors up to about $10^5$.}
\end{figure}

\begin{figure}
    \centering
    \includegraphics[width=\defaultwidth]{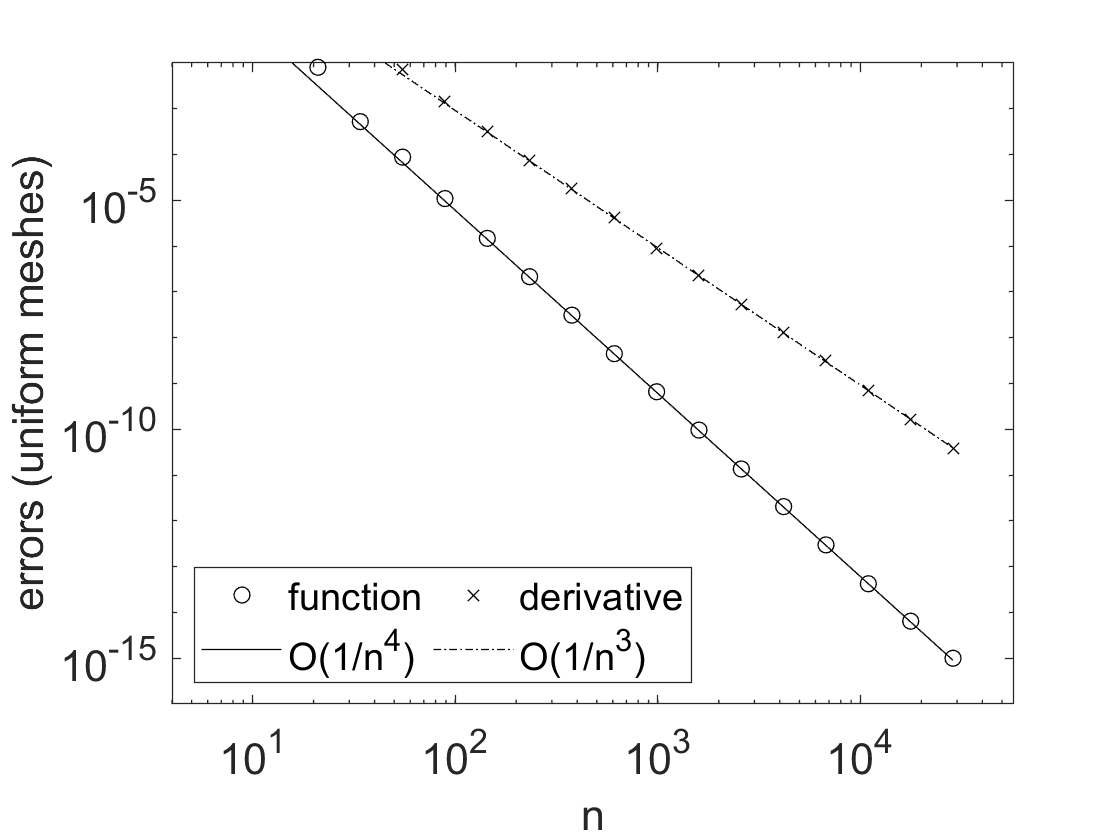}
    \caption[4th-order uniform mesh compact finite difference maximum error]{\label{fig:UniformMeshesRunge} Error in fourth-order compact cubic spline and its derivative for the Runge function on uniform meshes on $[-1,1]$.}
\end{figure}

\begin{remark}
The tridiagonal system for splines or compact finite differences is always nonsingular for positive mesh widths, as we proved.  Moreover, it is often well-conditioned (as we will see in the next section).
As always with finite-differences, however, 
rounding errors will make
their presence felt for small enough interval widths.  This is because differentiation is \textsl{infinitely ill-conditioned} (a fact which is well-known, but rarely discussed as such: see the
exposition in chapter~11 of~\cite{corless2013graduate} for one
instance).  
\end{remark}
\section{Explanation\label{sec:explanation}}
Why are the equations for a cubic spline, which guarantee twice-continuous differentiability for the piecewise cubic interpolation of the data, the same equations as those for (equally-spaced) fourth-order compact finite differences? The problems being solved by the two algorithms are different!

Saying it again: in the cubic spline case, we are constructing separate cubic polynomials on each subinterval $\tau_k \le t \le \tau_{k+1}$.  To guarantee twice continuous differentiability, we specify that the interior derivatives $\rho_{i,1}$ for $1 \le i \le n-1$ should satisfy certain equations.  But why should those equations give us fourth-order accurate derivatives in the equally-spaced case, but only third-order accurate derivatives for locally nonuniform meshes?

The most satisfactory explanation I have found so far is that the partial fraction decomposition of
\begin{equation}
    \frac{1}{(z-\tau_{k-1})^2(z-\tau_k)^3(z-\tau_{k+1})^2}
\end{equation}
has zero coefficient\footnote{This is similar to the contour integral explanation of the well-known extra order of accuracy of Simpson's rule: we fit three points, but in the equally-spaced case we get exact fit for a cubic, not just a quadratic.} of $1/(z-\tau_k)^2$ if (and only if) $\tau_{k+1}-\tau_k = \tau_k-\tau_{k-1}$.    Explicitly, here, if $\tau_{k+1}-\tau_k = \tau_k-\tau_{k-1} = h$ (as we will see without loss of generality we may take $h=1$) we have the following
\begin{theorem}
Let $t$ and $h$ be fixed complex numbers, $h \ne 0$, and $p(z)$ be a polynomial in the variable $z$.
If the polynomial $p(z)$ has degree at most $5$, then the values of $p(t-h)$, $p'(t-h)$, $p(t)$, $p(t+h)$, and $p'(t+h)$ determine $p''(t)$ uniquely.  Simultaneously, if $p(z)$ has degree less than or equal to $4$, those same values determine $p'(t)$ uniquely.
\end{theorem}
\begin{proof}
For intelligibility we may without loss of generality replace $z$ with $\theta$ where $z = t + \theta h$, so $-1 \le \theta \le 1$. Then we have the partial fraction decomposition
\begin{equation}
    \frac{1}{(\theta+1)^2\theta^3(\theta-1)^2} = -\frac{\sfrac{1}{4}}{(\theta+1)^2} -\frac{1}{\theta+1} + \frac{1}{\theta^3} + \frac{2}{\theta} + \frac{1}{\theta-1} + \frac{\sfrac14}{(\theta-1)^2}\>,
\end{equation}
which (as noted in the preamble to the theorem) has no $1/\theta^2$ term. By a standard argument, the contour integral
\begin{equation}
   \frac{1}{2\pi i} \oint_C \frac{p(\theta)}{(\theta+1)^2\theta^3(\theta-1)^2}\,d\theta = 0 
\end{equation}
for all polynomials of degree~$5$ or less if the contour~$C$ contains $-1$, $0$, and $1$. Therefore the partial fraction decomposition and Cauchy's integral formula gives
\begin{equation}
    0 = -p(-1) + 2p(0) - p(1) + \frac{1}{4}\left( p'(1) - p'(-1)\right) + \frac{1}{2}p''(0)\>.
\end{equation}
This means that if we know the value of the function at three consecutive points (which for a spline would take two cubic polynomials to fit), with equal width intervals on either side of the midpoint $0$, and we know the first derivatives at the two endpoints, then the value of the \textsl{second} derivative at the midpoint $0$ is determined irrespective of the value of the first derivative $p'(0)$.
Since these same quantities determine $p'(0)$ by equation~\eqref{eq:compactequispaced1} for polynomials of degree at most four, the proof is complete.
{\hfill$\natural$}
\end{proof}
Undoing the nondimensionalization and translating this into a finite-difference formula on the three points $\tau_{k-1}$, $\tau_k$, $\tau_{k+1}$ where $h = \tau_{k+1}-\tau_k = \tau_k-\tau_{k+1}$, we have
\begin{equation}
    f''(\tau_k) = 2\frac{ f(\tau_{k+1})-2f(\tau_k)+f(\tau_{k-1})}{h^2} - \frac{f'(\tau_{k+1})-f'(\tau_{k-1})}{2h}  + \frac{1}{360}f^{(6)}(0) h^4 + O(h^8)\>. \label{eq:secondderfourth}
\end{equation}
Notice the familiar second difference formula and the centered difference formula (using $f'$) appearing.  That this combination is fourth-order accurate was new to me but hardly seems surprising; doubtless the formula is in one book of finite differences or another, perhaps~\cite{MilneThomson(1951)}, \cite{Collatz(1966)}, \cite{jordan1965calculus}, or even~\cite{boole1860treatise}.  The error term in equation~\eqref{eq:secondderfourth} was computed in Maple, and indicates that the formula is exact if the sixth derivative is identically zero, \textsl{i.e.} if $f(z)$ is a polynomial of degree at most five. 
Since the value of the derivative at $\tau_k$ is not used, we can see that imposing a fourth-order accurate value of the derivative there must respect this constraint on the second derivative.

An alternative view is that in the equally-spaced case we may force an approximate value of $f'(\tau_k)$ that is in $O(1)$ error without altering this finite-difference value of $f''(\tau_k)$. 

That this symmetry forces the value of the second derivative to be a particular value for all fourth-order accurate derivatives still seems somewhat surprising.  The key seems to be the extra accuracy, allowing the second derivative to be determined, that is permitted by the zero residue in the (locally) equally-spaced case.



\section{Numerical tests}
I wrote a small didactic (and therefore relatively inefficient) \ML\ program that I called \texttt{compactcubic}, based on the codes \texttt{pchiptx} and \texttt{splinetx} from~\cite{Moler(2004)}, and which I will make available, either in the code repository for~\cite{corless2013graduate} or some other convenient place.  

Rather than leave the interpolant in the Hermite interpolational basis, which I prefer, for ease of comparison I computed the local monomial basis coefficients from the derivatives, much as is done in \texttt{splinetx}.  Indeed the formul\ae\ are very similar. I added the ability to output the first derivative. 
I also examine the $2$nd derivative, which turns out to be informative, though normally clearly departing from the $2$nd derivative of the underlying function. That this is not continuous for nonuniform meshes is demonstrated in figure~\ref{fig:ChebyshevRungeSecond7}.

I ran this on several functions, notably the Runge example $f(x) = 1/(1+25x^2)$ and the signum function (\texttt{sign} in \ML), using uniform meshes and nonuniform meshes of random widths.  The nonuniform meshes of random widths were constructed by first generating the random widths with \texttt{rand(1,n)} and then constructing the nodes by use of \texttt{cumsum}.

Some of the results are given in figures~\ref{fig:RandomMeshesRunge} and~\ref{fig:UniformMeshesRunge}.  In these tests, the programs appear to have behaved satisfactorily.  The condition numbers of the tridiagonal matrices depend on the node family used.  Equally-spaced nodes generate a condition number essentially constant; random meshes generate condition numbers that depend on the mesh spacing, but not the dimension.   The $O(1/n^4)$ behaviour of the errors on uniform grids and the $O(1/n^3)$ behaviour of the errors of the derivative are shown.


\begin{figure}[ht]
\begin{center}
\includegraphics[width=\defaultwidth]{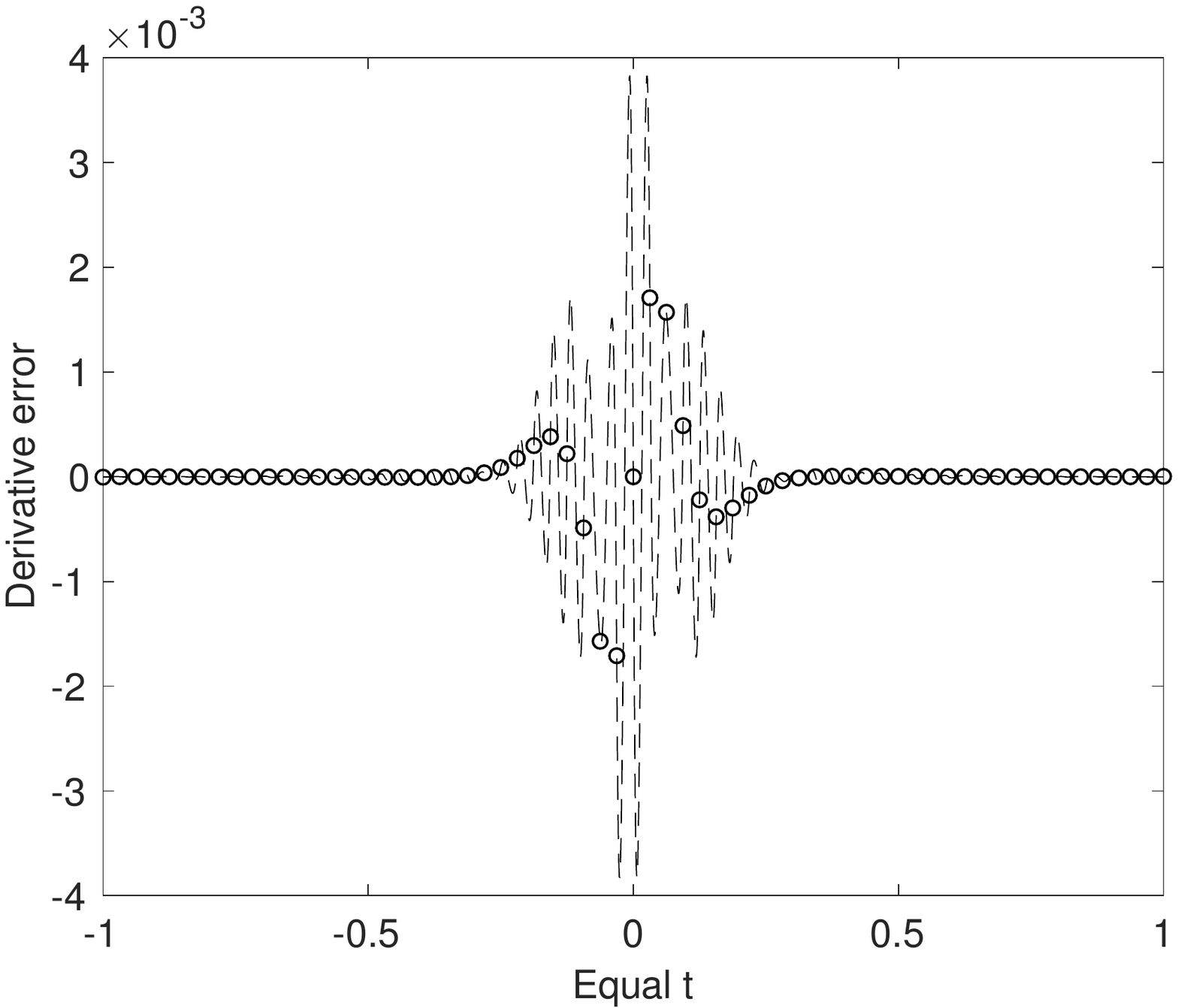}
\caption[equal mesh derivative error]{\label{fig:equal64RungeDerivative} Plot of the error in the fourth-order uniform mesh ($n=64$) compact finite difference \textsl{derivative}, on the Runge example $f(x) = 1/(1+25x^2)$, on the interval $[-1,1]$.  Solid line is the error between mesh points.  The error is largest in the middle, and believably $O(h^3)$.  The error in the interpolant itself (not the derivative as we have it in this figure) is $O(h^4)$. }
\end{center}
\end{figure}
\begin{figure}[ht]
\begin{center}
\includegraphics[width=\defaultwidth]{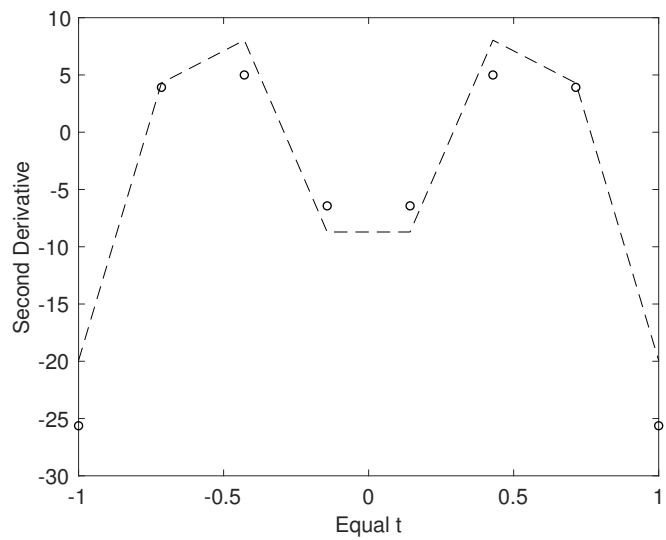}
\caption[Second derivative error]{\label{fig:equal7Runge2der} Plot of the error in \textsl{second} derivative of the  cubic Hermite interpolant. We use $n=7$ subintervals on $[-1,1]$ for the Runge example $f(x) = 1/(1+25x^2)$.  The graph shows clearly that the second derivatives of the spline are continuous at the nodes, although as expected they do not match the second derivatives at the nodes computed by the repeated application of \texttt{vcompact4.m}.}
\end{center}
\end{figure}

\begin{figure}[ht]
\begin{center}
\includegraphics[width=\defaultwidth]{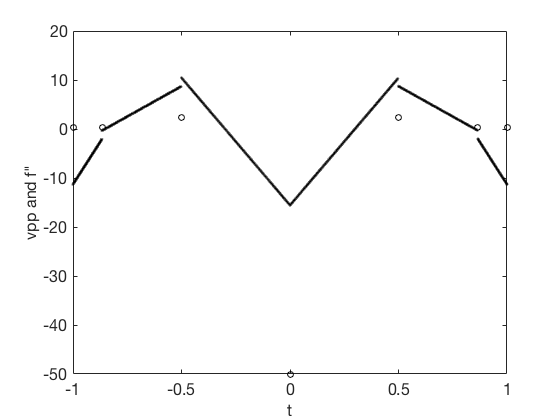}
\caption[Second derivative error]{\label{fig:ChebyshevRungeSecond7} Plot of the error in \textsl{second} derivative of the  compact cubic interpolant. We use $n=7$ subintervals with nonuniform (in fact, Chebyshev) nodes on $[-1,1]$ for the Runge example $f(x) = 1/(1+25x^2)$.  The graph shows clearly that the (linear) second derivatives of the cubic spline are \textsl{not} continuous at the nodes, as they would have been for equally-spaced nodes.  Circles are the exact second derivatives at the nodes.  Accuracy is about as expected with only seven nodes and errors are clearly visible.}
\end{center}
\end{figure}
\begin{remark}
Somewhat annoyingly, the plain vanilla \ML\ built-in spline function and piecewise polynomial evaluation routines do not have automatic access to evaluating the derivatives of the piecewise functions. Of course one could write one's own spline differentiator using \texttt{unmkpp} and \texttt{mkpp} (it only takes two lines of code), but bookkeeping other people's choices of piecewise polynomial representation is not always trivial, and the user might not notice the solution given in the help page for \texttt{mkpp}.  I believe that the built-in ability to take derivatives is always helpful\footnote{As a colleague points out, piecewise polynomial differentiation functionality \textsl{is} available in the \texttt{CurveFitting} package, which can be downloaded and installed (at my institution, at no extra cost; on the Mathworks page it offers the possibility of a free trial).}.
But for this paper, to make comparisons, instead of using \texttt{mkpp} and \texttt{unmkpp} I modified the \texttt{splinetx.m} function to return derivatives.  

As predicted, for equally-spaced nodes the derivatives were the same.  For Chebyshev nodes the derivatives were nearly indistinguishable ($r-s$ is typically quite small); for Fibonacci nodes there was a difference, but not much. 
\end{remark}

\subsection{Contour Integrals and Higher Derivatives \label{sec:higher}}
The contour integral technique used above is of course classical.  I learned it from John C.~Butcher, who used the following elegant variation in~\cite{Butcher(1967)} and later in~\cite{Butcher(2010)} to compute higher derivatives. 
The idea is to multiply by another factor, which I call the Butcher factor and denote by $B(z)$, that zeros out the unneeded and interfering residues.  As an example, suppose that we wish to find a formula for the second derivative of a cubic spline.  We then wish a rational function with denominator $(z-\tau_n)^2(z-t)^3(z-\tau_{n+1})^3$, so that we will pick up a term of the form $1/(z-t)^3$ which will supply us with $p''(t)/2!$ by the Cauchy Integral Formula.  But this approach will also generate a term for $p(t)$ and another for $p'(t)$, all mixed in together; so we introduce the Butcher factor
\begin{equation}
    B(z) = b_0 + b_1 z + b_2 z^2
\end{equation}
and choose the coefficients $b_0$, $b_1$, and $b_2$ (not all zero) so that the residues of $1/(z-t)$ and $1/(z-t)^2$ are zero.  This gives us two linear equations in three variables to solve, which we can do straightforwardly. One reasonably presentable solution is to put $t=\tau_n+ h\phi$, $z = \tau_n + h\theta$ and $\tau_{n+1}= \tau_n + h$ for neatness, and then
\begin{equation}
    B(\theta) = \left( \left( 3\,\phi-2 \right) {\phi}^{3}\right) -\left(2\,{\phi}
^{2} \left( 4\,\phi-3 \right)\right)\,\theta + \left( 6\,{\phi}^{2}-6\,\phi+1 \right)\,{\theta}^{2}
\end{equation}
which gives in turn
\begin{equation}
\frac{B(\theta)}{\theta^2(\theta-\phi)^3(\theta-1)^2} =
{\frac {-3\,\phi+1}{ \left( \theta-1 \right) ^{2}}}+{\frac {-6\,\phi+3
}{\theta}}+{\frac {6\,\phi-3}{\theta-1}}+{\frac {-3\,\phi+2}{{\theta}^
{2}}}+ \frac{1}{\left( \theta-\phi \right)^3}
     \>.
\end{equation}
Notice that as intended the coefficients of $1/(z-t)$ and $1/(z-t)^2$ (equivalently $1/(\theta-\phi)$ and $1/(\theta-\phi)^2$) are zero.
This in turn gives (after the same contour argument as before) that 
\begin{equation}
    \frac{2-3\phi}{h}p'(\tau_n)
 +\frac{1-3\phi}{h}p'(\tau_{n+1})+ \frac{p''(t)}{2!}
    + \frac{6\phi - 3}{h^2}p(\tau_{n+1}) + \frac{3-6\phi}{h}p(\tau_n) = 0\>,
\end{equation}
if the degree of $p(z)$ is at most three. Rearranging, we get
\begin{equation}
    p''(\tau_n+\phi h) = \frac{2(3\phi-2)}{h} p'(\tau_n) + \frac{2(3\phi-1)}{h}p'(t_{n+1}) + \frac{6(1-2\phi)}{h}\left( \frac{p(t_{n+1})-p(\tau_n)}{h}\right)\>.
\end{equation}
To use the same idea to find a compact formula for the second derivative given only function values on equally-spaced nodes, the idea is to search for a Butcher factor $B(z)$ that makes the coefficients of $(\theta+1)^{-2}$, $\theta^{-2}$, and $(\theta-1)^{-2}$ zero in the partial fraction expansion of 
\begin{equation}
\frac{b_0 + b_1 \theta + b_2 \theta^2}{(\theta+1)^3 \theta^3 (\theta-1)^3}\>.
\end{equation}
Here we have taken advantage of the equal spacing of the mesh to put $z = \tau_k + \theta h$ again. A short computation shows that $B(\theta) = 9\theta^2-5$ gives
\begin{equation}
\frac{9 \theta^2-5}{(\theta+1)^3 \theta^3 (\theta-1)^3} = \frac{-3}{\theta+1} + \frac{6}{\theta} - \frac{3}{\theta-1} + \frac{\sfrac12}{(\theta+1)^3} + \frac{5}{\theta^3} + \frac{\sfrac12}{(\theta-1)^3}\>.
\end{equation}
After using the Cauchy Integral Formula again, we find the following fourth-order formula
\begin{equation}
p''(\tau_n-h) + 10 p''(\tau_n) + p''(\tau_n+h) = \frac{12}{h^2}\left( p(\tau_n-h) - 2p(\tau_n) + p(\tau_n+h)\right)\>.
\end{equation}
The Cauchy formula already tells us that this is exact for polynomials of degree at most $5$. A separate Taylor series expansion (trivial in a computer algebra system) gives that the error term for functions not polynomials of degree at most five is asymptotically, as $h\to 0$,
\begin{equation}
\frac{h^4}{20} p^{(6)}(\tau_n) + \frac{11h^6}{5040} p^{(8)}(\tau_n) + O(h^8)\>.
\end{equation}

\subsection{Notes and further reading\label{sec:other}}

The books~\cite{deBoor(1978)} and~\cite{jordan1965calculus} are each rich sources of the history of their respective subjects. Much of this section comes from them. 
The word ``spline'' comes from the same root as the English word ``splinter'', meaning a thin fragment of wood; such wooden pieces gave rise to a flexible mechanical device that was used, apparently in shipbuilding, to draw smooth curves between fixed points.  I.~J.~Schoenberg in~\cite{Schoenberg1946contributions} chose the name ``spline'' for the piecewise cubic polynomials computed in this fashion because they minimize an approximation of the strain energy that the original flexible instrument minimized.  The modern terminology, and in particular the difference between a cubic spline, a \texttt{pchip}, and a true cubic Hermite interpolant (where the derivative values $f'(\tau_k)$ are also known exactly and not chosen by the algorithm), is somewhat confusing, sometimes even to experts.

I was unable to trace the first use of the word ``compact'' for finite differences that are defined as a relation amongst function and derivative values on a small set of mesh points.  Some authors use instead ``optimal'', conveying that the maximum accuracy possible is attained given the constraints on the number of mesh points.  Collatz used the term ``Mehrstellenverfahren'', which is still used.

This paper is concerned with the connection between these two concepts. To explain this connection I used the contour integral approach.  As is very well known, this is not the only approach to finite differences, compact or otherwise.  

One standard approach is simply to use brute Taylor expansion of an ansatz with undetermined coefficients, then set all the desired coefficients of powers of the step size to zero, and brutally solve for the unknowns.  Doing this once and for all for a fixed (possibly symbolic) grid in a computer algebra system is perfectly straightforward and perfectly useful, and I have used this brutal method in a few papers such as~\cite{Zhao(2007)}, and before that in~\cite{corless1996symbolic}.  As pointed out in~\cite{fornberg1998classroom} this approach cannot easily be used in a purely numerical environment because the linear systems are often quite ill-conditioned, and more so for larger formul\ae; but in a computer algebra system the arithmetic can be done exactly (if rational), or at whatever precision desired, and the expense of doing so is amortized over the life of the \textsl{use} of the formul\ae.  

Also as pointed out in~\cite{fornberg1998classroom}, and in~\cite{corless2013graduate}, the \textsl{use} of the formul\ae\ (however obtained) is subject to numerical difficulty; this is because differentiation is infinitely ill-conditioned.  But the brute-force approach of generating formul\ae\ by computer algebra systems works and gets formul\ae\ as accurate as any method does\footnote{Well, maybe there is still room for blunders.  The paper~\cite{keller1978symbolic} used Macsyma to generate weights for finite difference formul\ae.  The paper~\cite{fornberg1988generation} points out that the tables in~\cite{keller1978symbolic} contained ``both isolated and systematic errors''. Certainly while I have been as careful as I can for this paper, and I have taken out a lot of typos, I have probably introduced some in transferring formul\ae\ from Maple to \LaTeX.}.

I believe that the contour integral approach is better at \textsl{explanation}, however. Here, of course, we are trying to understand the connection of cubic splines to compact finite differences, and the contour integral approach supplied a satisfactory explanation.  Contour methods are often used in this way.  According to~\cite{weideman2005pade} contour integrals were used by the great number theorist K.~Mahler to explore Pad\'e approximations to the logarithms of algebraic numbers in~\cite{mahler1953approximation}.
The connection to Pad\'e approximation (and indeed compact finite difference formul\ae\ are often called ``Pad\'e methods'') arises because of another, perhaps even more elegant method for constructing finite difference formul\ae, the \textsl{operator method}. One version of the operator method, derived by applying difference operations to the exponential function, is used to great effect in~\cite{fornberg1998classroom} to compute weights in many finite-difference formul\ae\ and to give general recurrence relations for doing so on arbitrary grids. In older works, such as~\cite{fox1957numerical}, we see a more ``formal'' treatment of operator methods; here `formal' means that rigour is \textsl{not} maintained throughout the computation (a curious usage of this word, that often causes puzzlement, meaning only that computations are just carried out by `form' and not worrying until later whether they were correct). 
The operator method was recently blogged about in
{Scientific American} by~\cite{Lamb:2019}. I gave a talk on this method at a CAIMS meeting in Winnipeg in 2005. It seems to be common in early works on finite differences, such as~\cite{MilneThomson(1951)}, and is at least a hundred and sixty years old, being present in~\cite{boole1860treatise}; in fact, according to~\cite{jordan1965calculus} the method is due to Lagrange himself. 

The operator method rests on an analogy: the shift operator $E_h(f)(t) = f(t+h)$ can be expressed in terms of the derivative operator $hD$ by Taylor's theorem for analytic functions as
\begin{equation}
E_h(f)(t) = f(t+h) = \sum_{k\ge 0} \frac{h^k}{k!} D^{(k)}(f)(t),
\end{equation}
where the symbol $D^{(k)}$ means repeated application of the differentiation operator. By analogy with multiplication, this is written as $E_h = \exp( hD )(f)(t)$, which though likely familiar to the reader from various contexts is actually a remarkable leap to a definition of an exponential of an operator. Formally solving for $D$ gives $D = \ln(E_h)(f)(t)/h = \ln(1+\Delta)(f)(t)/h$, motivating Pad\'e approximations to the logarithm near $1$ (or the inverse sinh function near $0$ because analogously $\exp(hD/2)-\exp(-hD/2) = 2\sinh(hD/2)$ applied to $f(t)$ gives $f(t+h/2)-f(t-h/2)$ and so the inverse sinh function expresses $D$ in terms of compositions of central differences). The resulting Pad\'e approximations then give us compact finite difference formul\ae.

The paper~\cite{weideman2005pade} extends this idea to Hermite-Pad\'e approximations, and solves several general problems associated with these.
Further, that paper details an open problem about explicit descriptions of arbitrary width compact formul\ae\ for the \textsl{second} derivative; I do not know if the Butcher factor idea for contour integration can be used to attack that open problem, but John C.~Butcher \textsl{did} use his approach to solve another general degree interpolation problem in quadrature (we detailed his solution in~\cite{Butcher(2010)}), so it might be possible. I note that one can establish that a necessary condition for the unwanted residues to vanish if the degree of $B(z)$ is to be no more than $n$ is that at each node $\tau_k$, $0 \le k \le n$, the Butcher factor must have 
\begin{equation}
 B'(\tau_k)  - 3\left( \sum_{\substack{j=0\\j\ne k}}^n \frac{1}{\tau_k-\tau_j} \right)B(\tau_k) = 0\>,
\end{equation}
which implies that $B(z)$ can be identified in the Lagrange basis by its values on the nodes by finding a vector in the nullspace of $\mat{D}-3\diag(s_0, s_1, \ldots s_n)$ where $s_k = \sum_{\substack{j=0\\j\ne k}}^n \frac{1}{\tau_k-\tau_j}$ and $\mat{D}$ is the differentiation matrix on the nodes (for differentiation matrices, see e.g.~\cite{amiraslani2018differentiation}).
I leave this to future work.  
\subsection{Higher dimension}
Compact finite differences are widely used in problems having more than one dimension. An example on a regular grid is~\cite{Rokicki(1995)}, and on irregular grids is~\cite{rokicki1999higher}. Automatic computation of such formul\ae\ for scattered nodes in more than one dimension are considered in more detail in the University of Western Ontario Ph.D.\ thesis of Jichao Zhao (2006) and a subsequent paper~\cite{Zhao(2006)} and a Maple Share Library package released in that same year updating the original 1994 package.  Other researchers have also studied the problem, for example~\cite{wright2006scattered} who use radial basis functions and were likely the first to do this in full generality.

\section{Concluding remarks, and a marriage of convenience}
The usual methods of dealing with the two extra
degrees of freedom of a cubic spline have always seemed somewhat unnatural to me.  As has been demonstrated with many psychological studies, the extra choice seems unwelcome.  What I recommend here is to use compact finite differences to not only compute the derivatives at the interior nodes, but also at the endpoints.  This is very similar in spirit to the recommendation of~\cite{swartz1972error}, except here we are using a compact finite difference formula to give fourth-order accuracy, where they recommended a third-order accurate direct formula.  The method recommended here requires essentially no more effort, being simply a modification of the tridiagonal system needed to find the spline derivatives, and has higher order accuracy for the derivatives at the edges.  This new compact method, with its positive definite (totally nonnegative) matrices, therefore seems as though it will be as good a method as any, and might be psychologically more satisfying in that the derivatives at the end are determined by the data.  Of course, they are also determined by the data, in a sense, if instead the ``not-a-knot'' condition is used, but that makes a qualitative difference with the two subintervals at each end.
Whether it makes a \textsl{practical} difference is another matter.  

I have not made extensive comparisons with \ML\ 's built-in \texttt{spline}, for example (but the few experiments I have done show that the behaviour is at least similar). The fact that the Not-A-Knot condition generates an $O(h^3)$ error in the derivative at the end may or may not be significant.
Nor have I made a choice between the $2+\sqrt{3}$ method---which needs at least $n=4$ so therefore $5$ nodes and has an analytic factoring for uniform meshes---and the $h_2(h_2+h_3)/((h_1+h_2)(h_1+h_2+h_3))$ method, which also needs at least $5$ nodes because its four-by-four matrix is singular; in the first case the condition number of the matrix is smaller, but the truncation error is larger (by a similar factor). My only experience with this in practice is with the uniform mesh case, where the efficiency of the analytic factoring is noticeably useful. More experimentation is necessary to learn if the variable mesh formula is truly useful.

The method has been coded in \ML\ as the program \texttt{compactcubic.m} and will be made available at the code repository for~\cite{corless2013graduate}, namely \url{nfillion.com/coderepository}.  

It seems that this approach gives a potentially reasonable alternative to splines, in that its derivative can be expected to be more accurate \textsl{at the nodes} for nonuniform meshes, and one does not have to make a choice about what to do at the endpoints.  That first caveat, ``at the nodes'' is important.  The error in piecewise cubic interpolation means that the derivatives \textsl{between the nodes} can only be $O(h^3)$ accurate.  This weakens the case for the compact cubic interpolant proposed here, it's true; away from the nodes, one doesn't see more accurate derivatives.

This approach needs a name to distinguish it from ``spline'', ``cubic Hermite'', and ``pchip'' (itself a misleading name for the interpolant whose derivatives are chosen to preserve qualitative features, at the cost of even less accuracy in the derivative).  I propose the name ``compact cubic interpolant'' for this construction.  It also needs an efficient and stable implementation; probably conversion to the local monomial basis is the most efficient, as is done here, but there are likely other tricks needed to ensure that the computation is as stable as possible.

I realize that I have not made a strong case for using this method---cubic splines are pretty good, after all, even if they occasionally wiggle too much (which is what \texttt{pchip} is for)---but the exploration of this alternative has yielded \textsl{some} advantages for the compact cubic method: file-it-and-forget-it treatment of the edges, and in theory more accurate derivatives at the nodes.  If one is using compact finite differences for some other purpose anyway (perhaps as part of solving a PDE by the method of lines) then the analysis in this paper may help you to choose to use the compact cubic interpolant.

I remark that totally positive (nonnegative) matrices occur in various interpolation problems---see~\cite{MARCO201763} for an example.

The other result of this paper, namely an explanation, using contour integrals and residues, of the apparent coincidence that cubic splines give the same derivatives that the fourth order Pad\'e compact formula does, is only one such explanation.  This coincidence can also be explained by noting first (as can be found in Appendix A of~\cite{fornberg2015primer} for instance) that a cubic spline over $x_0-h, x_0, x_0+h$ is the sum of a pure cubic and a multiple of $|x-x_0|^3$. Then since the compact formula is exact for pure cubics, one only needs to think about its behaviour on $|x-x_0|^3$; but this is an even function, and a moment's reflection shows that~\eqref{eq:spline} is exact for even functions, giving zero derivative at the centre; so since the equation is exact for each, it is exact for their sum, and hence for splines. Thus this coincidence might not be so surprising for one ``skilled in the art''.  This argument works only in the case of real nodes, of course, whereas the residue argument works over~$\mathbb{C}$, but the overwhelming majority of spline interpolation takes place over~$\mathbb{R}$, so that doesn't matter much. Still, I think that this coincidence has hardly been \textsl{noticed} in the literature; that it can be explained several ways is perhaps not a surprise. I do like the residue explanation in preference to the $|x-x_0|^3$ explanation because zero residues are connected to the extra degree of accuracy obtained in Simpson's Rule, and connected to the Butcher method of finding higher derivative formul\ae\ and solving the Birkhoff interpolation problem.  This, however, is clearly a matter of taste: I am certain that there are many who would prefer the explanation using $|x-x_0|^3$.

A referee points out that I have not reviewed the spline literature thoroughly in this paper, omitting reference even to the classic book~\cite{schumaker2007spline}. In my defence I point out that this paper is already too long.


\end{document}